\documentclass{amsart}
\usepackage{tikz}
\usepackage{tikz-cd}
\usepackage{scalefnt}
\usepackage[latin1]{inputenc}
\usepackage[T1]{fontenc}
\usepackage{amsmath,amssymb,amsthm,bm}
\usepackage{graphicx}
\usepackage{hyperref}
\usepackage{upgreek}
\usepackage{scalerel}

\DeclareMathOperator*{\bigcdot}{\scalerel*{\cdot}{\bigodot}}

\newtheorem{theorem}{Theorem}[section]
\newtheorem{lemma}[theorem]{Lemma}

\theoremstyle{definition}
\newtheorem{definition}[theorem]{Definition}

\newcommand{\config}{\mathcal{C}}
\newcommand{\co}{\colon\thinspace}
\newcommand{\abs}[1]{\left\lvert #1 \right\rvert}

\DeclareMathOperator{\key}{key}
\DeclareMathOperator{\code}{code}

\DeclareMathOperator{\cell}{cell}
\DeclareMathOperator{\poset}{poset}
\DeclareMathOperator{\symbols}{\mathcal{A}}
\DeclareMathOperator{\partition}{part}

\DeclareRobustCommand{\stirling}{\genfrac[]{0pt}{}}

\newcommand{\R}{\mathbb{R}}

\newcommand{\Z}{\mathbb{Z}}
\newcommand{\hequiv}{\simeq}

\begin{document}

\author{Hannah Alpert}
\author{Matthew Kahle}
\thanks{M.K.\ thanks IAS for hosting him as a member in 2010--11, and for several visits since then. He gratefully acknowledges the support of a Sloan Research Fellowship, NSF \#DMS-1352386, and NSF \#CCF-1839358.  H.A.\ is supported by the National Science Foundation under Award No.\ DMS 1802914.}
\author{Robert MacPherson}
\thanks{All three authors thank ICERM for hosting them during the special thematic semester ``Topology in Motion'' in Autumn 2016.}
% \blfootnote{\textit{Keywords: configuration spaces, phase transitions} }
% \blfootnote{\textit{2020 Mathematics Subject Classification:}55R80, 	82B26 }
\title{Configuration spaces of disks in an infinite strip}



\maketitle

\begin{abstract}
We study the topology of the configuration spaces $\config(n,w)$ of $n$ hard disks of unit diameter in an infinite strip of width $w$. We describe ranges of parameter or ``regimes'', where homology $H_j [\config(n,w)]$  behaves in qualitatively different ways.

We show that if $w \ge j+2$, then the homology $H_j[\config(n, w)]$ is isomorphic to the homology of the configuration space of points in the plane, $H_j[\config(n, \R^2)]$.  The Betti numbers of $\config (n, \R^2) $ were computed by Arnold \cite{Arnold14}, and so as a corollary of the isomorphism, $\beta_j[\config(n,w)]$ is a polynomial in $n$ of degree $2j$.

On the other hand, we show that if $2 \le w \le j+1$, then $\beta_j [ \config(n,w) ]$ grows exponentially with $n$. Most of our work is in carefully estimating $\beta_j [ \config(n,w) ]$ in this regime.

We also illustrate, for every $n$, the homological ``phase portrait'' in the $(w,j)$-plane--- the parameter values where homology $H_j [ \config(n,w)]$ is trivial, nontrivial, and isomorphic with $H_j [ \config(n, \R^2)]$. Motivated by the notion of phase transitions for hard-spheres systems, we discuss these as the ``homological solid, liquid, and gas'' regimes.

\end{abstract}

\section{Introduction}

We study the topology of configuration spaces $\config(n,w)$ of $n$ non-overlapping disks of unit diameter in an infinite strip of width $w$.

In other words, for non-negative integers $n, w$ we define
\begin{align*}
\config(n,w)  &= \{ (x_1, y_1, x_2, y_2, \dots, x_n, y_n) \in \R^{2n} : \\
&(x_i-x_j)^2 + (y_i-y_j)^2 \ge 1 \mbox{ for every } i \neq j, \mbox{ and } \\
& 1/2 \le y_i \le w -1/2 \mbox{ for every } i. \} 
\end{align*}
The coordinates $(x_i, y_i)$ are the centers of disk $i$, the inequalities $(x_i-x_j)^2 + (y_i-y_j)^2 \ge 1$ ensure that the disks have disjoint interiors, and the inequalities $1/2 \le y_i \le w -1/2$ ensure that the disks of unit diameter (or radius $1/2$) stay in the closed strip $0 \le y \le w$. These spaces generalize the well-studied configuration space of points in the plane, which we denote $\config(n, \R^2)$. 

Our main result describes the asymptotics for the Betti numbers $\beta_j [ \config(n,w)]$, for fixed $j$ and $w$, as $n \to \infty$. Our results do not depend on the choice of coefficient field for the homology.  We use the notation $f  \asymp g$ to indicate that there exist positive constants $c_1, c_2$ such that $$c_1 g(n) \le f(n) \le c_2 g(n)$$ 
for all sufficiently large $n$. In the following, the implied constants depend on $j$ and $w$ but not on $n$.

\begin{figure}
\centering
% \scalefont{2.5}
 \scalebox{0.85}{
\begin{tikzpicture}
\large
\draw[opacity=0.5] [->] (-0.5,0.5) -- (10,0.5);
%\draw [opacity=0.25] (0,0.5) -- (8,0.5);
\draw [opacity=0.25] (-0.5,1.5) -- (10,1.5);
\draw [opacity=0.25] (-0.5,2.5) -- (10,2.5);
\draw [opacity=0.25] (-0.5,3.5) -- (10,3.5);
\draw [opacity=0.25] (-0.5,4.5) -- (10,4.5);
\draw [opacity=0.25] (-0.5,5.5) -- (10,5.5);
\draw [opacity=0.25] (-0.5,6.5) -- (10,6.5);
\draw [opacity=0.25] (-0.5,7.5) -- (10,7.5);
\draw [opacity=0.25] (-0.5,8.5) -- (10,8.5);

\draw[opacity=0.5] [->] (-0.5,0.5) -- (-0.5,9);
\draw [opacity=0.25](0.5,0.5) -- (0.5,9);
\draw[opacity=0.25] (1.5,0.5) -- (1.5,9);
\draw [opacity=0.25](2.5,0.5) -- (2.5,9);
\draw[opacity=0.25] (3.5,0.5) -- (3.5,9);
\draw [opacity=0.25](4.5,0.5) -- (4.5,9);
\draw [opacity=0.25](5.5,0.5) -- (5.5,9);
\draw [opacity=0.25](6.5,0.5) -- (6.5,9);
\draw [opacity=0.25](7.5,0.5) -- (7.5,9);
\draw [opacity=0.25](8.5,0.5) -- (8.5,9);
\draw [opacity=0.25](9.5,0.5) -- (9.5,9);

 \node at (10.5,0.5) {\bf  $w$};
 \node at (-0.5,9.5) {\bf  $j$};

\small
\node at (-1/2,17/2) {$0$};
\node at (1/2,17/2) {$0$};

\node at (-1/2,1/2) {$0$};
\node at (-1/2,3/2) {$0$};
\node at (-1/2, 5/2) {$0$};
\node at (-1/2, 7/2) {$0$};
\node at (-1/2, 9/2) {$0$};
\node at (-1/2, 11/2) {$0$};
\node at (-1/2, 13/2) {$0$};
\node at (-1/2, 15/2) {$0$};
\node at (-1,1/2) {0};
\node at (-1,3/2) {1};
\node at (-1,5/2) {2};
\node at (-1,7/2) {3};
\node at (-1,9/2) {4};
\node at (-1,11/2) {5};
\node at (-1,13/2) {6};
\node at (-1,15/2) {7};
\node at (-1,17/2) {8};
\node  at (1/2,1/2) {$\bf n!$};
\node at (1/2,3/2) {$0$};
\node at (1/2, 5/2) {$0$};
\node at (1/2, 7/2) {$0$};
\node at (1/2, 9/2) {$0$};
\node at (1/2, 11/2) {$0$};
\node at (1/2, 13/2) {$0$};
\node at (1/2, 15/2) {$0$};
\node at (-1/2,0) {0};
\node at (1/2,0) {1};
\node at (3/2,0) {2};
\node at (5/2,0) {3};
\node at (7/2,0) {4};
\node at (9/2,0) {5};
\node at (11/2,0) {6};
\node at (13/2,0) {7};
\node at (15/2,0) {8};
\node at (17/2,0) {9};
\node at (19/2,0) {10};
\node at (3/2,1/2) {$1$};
\node at (5/2,1/2) {$1$};
\node at (7/2,1/2) {$1$};
\node at (9/2,1/2) {$1$};
\node at (11/2,1/2) {$1$};
\node at (13/2,1/2) {$1$};
\node at (15/2,1/2) {$1$};
\node  at (3/2,3/2) {$ \bf 2^n n^2$};
\node  at (3/2, 5/2) {$ \bf 3^n n^4$};
\node  at (3/2, 7/2) {$ \bf 4^n n^6$};
\node  at (3/2, 9/2) {$ \bf 5^n n^8$};
\node  at (3/2, 11/2) {$ \bf 6^n n^{10}$};
\node  at (3/2, 13/2) {$ \bf 7^n n^{12}$};
\node  at (3/2, 15/2) {$ \bf 8^n n^{14}$};
\node   at (5/2, 5/2) {$ \bf 2^n n^3$};
\node   at (5/2, 7/2) {$ \bf 2^n n^5$};
\node   at (5/2, 9/2) {$ \bf 3^n n^6$};
\node   at (5/2, 11/2) {$ \bf 3^n n^8$};
\node   at (5/2, 13/2) {$ \bf 4^n n^9$};
\node   at (5/2, 15/2) {$ \bf 4^n n^{11}$};
\node   at (7/2, 7/2) {$ \bf 2^n n^4$};
\node   at (7/2, 9/2) {$ \bf 2^n n^6$};
\node   at (7/2, 11/2) {$ \bf 2^n n^8$};
\node   at (7/2, 13/2) {$ \bf 3^n  n^8$};
\node   at (7/2, 15/2) {$ \bf 3^n n^{10}$};
\node   at (9/2, 9/2) {$ \bf 2^n n^5$};
\node   at (9/2, 11/2) {$ \bf 2^n n^7$};
\node   at (9/2, 13/2) {$ \bf 2^n n^9$};
\node   at (9/2, 15/2) {$ \bf 2^n n^{11}$};
\node   at (11/2, 11/2) {$ \bf 2^n n^6$};
\node   at (11/2, 13/2) {$ \bf 2^n n^8$};
\node   at (11/2, 15/2) {$ \bf 2^n n^{10}$};
\node   at (13/2, 13/2) {$ \bf 2^n n^7$};
\node   at (13/2, 15/2) {$ \bf 2^n n^{9}$};
\node   at (15/2, 15/2) {$ \bf 2^n n^8$};
\node at (3/2,17/2) {$ \bf 9^n n^{16}$};
\node at (5/2,17/2) {$ \bf 5^n n^{12}$};
\node at (7/2,17/2) {$ \bf 3^n n^{12}$};
\node at (9/2,17/2) {$ \bf 3^n n^{10}$};
\node at (11/2,17/2) {$ \bf 2^n n^{12}$};
\node at (13/2,17/2) {$ \bf 2^n n^{11}$};
\node at (15/2,17/2) {$ \bf 2^n n^{10}$};
\node at (17/2,17/2) {$ \bf 2^n n^{9}$};

\node at (3/2,1/2) {$1$};
\node at (5/2,1/2) {$1$};
\node at (7/2,1/2) {$1$};
\node at (9/2,1/2) {$1$};
\node at (11/2,1/2) {$1$};
\node at (13/2,1/2) {$1$};
\node at (15/2,1/2) {$1$};
\node at (17/2,1/2) {$1$};
\node at (19/2,1/2) {$1$};

\node at (5/2,3/2) {$n^2$};
\node at (7/2,3/2) {$n^2$};
\node at (9/2,3/2) {$n^2$};
\node at (11/2,3/2) {$n^2$};
\node at (13/2,3/2) {$n^2$};
\node at (15/2,3/2) {$n^2$};
\node at (17/2,3/2) {$n^2$};
\node at (19/2,3/2) {$n^2$};

\node at (7/2,5/2) {$n^4$};
\node at (9/2,5/2) {$n^4$};
\node at (11/2,5/2) {$n^4$};
\node at (13/2,5/2) {$n^4$};
\node at (15/2,5/2) {$n^4$};
\node at (17/2,5/2) {$n^4$};
\node at (19/2,5/2) {$n^4$};

\node at (9/2,7/2) {$n^6$};
\node at (11/2,7/2) {$n^6$};
\node at (13/2,7/2) {$n^6$};
\node at (15/2,7/2) {$n^6$};
\node at (17/2,7/2) {$n^6$};
\node at (19/2,7/2) {$n^6$};

\node at (11/2,9/2) {$n^8$};
\node at (13/2,9/2) {$n^8$};
\node at (15/2,9/2) {$n^8$};
\node at (17/2,9/2) {$n^8$};
\node at (19/2,9/2) {$n^8$};

\node at (13/2,11/2) {$n^{10}$};
\node at (15/2,11/2) {$n^{10}$};
\node at (17/2,11/2) {$n^{10}$};
\node at (19/2,11/2) {$n^{10}$};

\node at (15/2,13/2) {$n^{12}$};
\node at (17/2,13/2) {$n^{12}$};
\node at (19/2,13/2) {$n^{12}$};

\node at (17/2,15/2) {$n^{14}$};
\node at (19/2,15/2) {$n^{14}$};
\node at (19/2,17/2) {$n^{16}$};
\end{tikzpicture}
}
\caption{Theorem \ref{thm:main} describes the rate of growth of $\beta_{j} [ \config(n,w)]$, for fixed $j$ and $w$, as $n \to \infty$. The results are up to a constant factor, e.g.\ 
$ \beta_{8} [ \config(n,3)]  \asymp  5^n n^{12}$.}

\label{fig:main}
\end{figure}
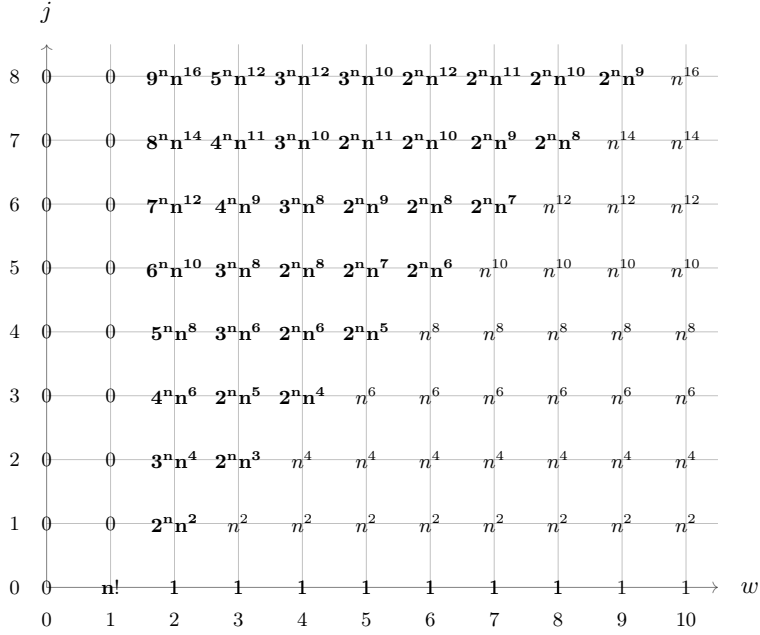

\begin{theorem}[Asymptotic rate of growth of the Betti numbers as $n \to \infty$] \label{thm:main} 
\noindent
\begin{enumerate}
\item If $w \ge 2$ and $0 \le j \le w-2$ then the inclusion map $i : \config(n,w) \to \config(n,\R^2)$ induces an isomorphism on homology 
$$i_*:\ H_j [ \config(n,w) ] \rightarrow H_j [ \config(n,\R^2)].$$
So if $n \to \infty$ then the asymptotic rate of growth is given by $$ \beta_j \left[ \config(n,w) \right] \asymp  n^{2j}.$$
\item If $w \ge 2$ and $j \ge w-1$ then write $j = q(w-1) + r$ with  $q \ge 1$ and $0 \le r < w-1$. Then we have that
$$\beta_j \left[ \config(n,w) \right] \asymp (q+1)^n n^{qw+2r}.$$
If $w=1$ and $j=0$, then $\beta_0=n!$.
\item If either $w=0$, or $w=1$ and $j \ge 1$, then $\beta_j = 0$.\\
\end{enumerate}
\end{theorem}

Configuration spaces of disks arise naturally as the phase space of a $2$-dimensional ``hard-spheres'' system, so are of interest in physics as well. See, for example, the discussion of hard disks in a box by Diaconis in \cite{Diaconis09}, and the review of the physics literature by Carlsson et al.\ in \cite{CGKM12}.

The topology of configuration spaces of particles with thickness has been studied earlier, for example in \cite{Alpert17}, \cite{BBK14}, \cite{Deeley11}, and \cite{KKLS16}, but so far, not much seems to be known. Some of this past work is also inspired in part by applications to engineering, for example motion planning for robots.

Inspired by the statement of Theorem \ref{thm:main}, we suggest the following definitions for ``homological solid, liquid, and gas'' regimes in the $(w,j)$ plane.

\begin{itemize}

\item We define the ``homological solid'' phase to be wherever homology is trivial.  The motivation for this definition is that one expects that in a crystal phase, things are fairly rigid and that the configuration space is simple.

\item We define the ``homological gas'' phase to be where homology agrees with the configuration space of points in the plane. In other words, through the lens of this homology group, the particles are indistinguishable from points, corresponding to the assumption of atoms acting as point particles in an ideal gas. Arnold \cite{Arnold14} showed that the Poincar\'e polynomial of $\config(n,\R^2)$ is given by 
$$\beta_0 + \beta_1 t + \dots + \beta_{n-1} t^{n-1} = (1+t)(1+2t) \dots (1+(n-1)t).$$
It follows that the Betti numbers are given by the unsigned Stirling numbers of the first kind.
$$\beta_j \left[ \config(n,\R^2) \right]  =\stirling{n}{n-j}.$$
For a self-contained overview of the homology and cohomology of $\config(n,\R^2)$, see Sinha \cite{Sinha13}.

One can use a standard recursive formula for Stirling numbers to write $ \stirling{n}{n-j}$ as a polynomial in $n$ of degree $2j$. See, for example, Section 1.3 of Stanley's book \cite{Stanley12}. Formulas for the first few Betti numbers are given by:
\begin{align*}
\beta_0 [ \config(n,\R^2)] &= 1& \\
\beta_1  [ \config(n,\R^2)]  &= \frac{n(n-1)}{2}  \\
\beta_2  [ \config(n,\R^2)]  &= \frac{(3n-1)n(n-1)(n-2)}{24}  \\
\beta_3  [ \config(n,\R^2)]  &= \frac{n^2(n-1)^2(n-2)(n-3)}{48} \\
\end{align*} 

\item Finally, we define the ``homological liquid'' phase to be everything else. This is the most interesting regime topologically, and we are somewhat surprised to find that there is a lot of homology. Another physical metaphor for the homological liquid regime, suggested to us by Jeremy Mason, is a turbulent fluid in a pipe.\\

\end{itemize}

Most of our work in this paper is in estimating the Betti numbers in the homological liquid regime. For lower bounds, we use the duality between the homology of $\config(n,w)$ and its homology with closed support. For upper bounds, we first prove that $\config(n,w)$ is homotopy equivalent to a cell complex $\cell(n,w)$, and then apply discrete Morse theory. \\

Some advantages of the definitions of homological solid, liquid, and gas include their simplicity, their generality, and being well defined for every finite $n$ and not merely asymptotically. All three regimes are already visible when $n=3$ and $j=1$. The following describes the shapes of the regimes for every $n$. We emphasize that the boundary between solid and liquid regimes is more interesting for finite $n$ than it appears to be in Theorem \ref{thm:main}.

\begin{theorem}[The phase portrait for every $n$] \label{thm:regimes} 

\noindent
\begin{enumerate}
\item (Gas regime.) If $w \ge 2$ and $0 \le j \le w-2$, then the inclusion map $i : \config(n,w) \to \config(n,\R^2)$ induces an isomorphism on homology 
$$i_*:\ H_j [ \config(n,w) ] \rightarrow H_j [ \config(n,\R^2)].$$
Moreover, if $w \ge n$ then $\config(n,w)$ is homotopy equivalent to $\config(n,\R^2)$.\\ 
\item (Liquid regime.) If $1 \le w \le n-1$ and $w-1 \le j \le n - \lceil n/w \rceil$ then $H_j ( \config( n, w) ) \neq 0$, but the inclusion map $i: \config(n,w) \to \config(n,\R^2)$  does not induce an isomorphism on homology
$$i_*:\ H_j [ \config(n,w) ] \rightarrow H_j [ \config(n,\R^2)].$$

\item (Solid regime.) If either $w=0$, or $w \ge 1$ and $j \ge n - \lceil n/w \rceil +1$, then 
$$H_j [ \config(n,w) ] = 0.$$

\end{enumerate}
\end{theorem}

 \begin{figure}

 \scalebox{0.3}{
 \begin{tikzpicture}[thick]
 \huge
 \draw [->] (0,0)--(24.5,0);
 \node at (25,0) {\scalefont{5} $w$};
 \draw [->] (0,0)--(0,24.5);
 \node at (-0.2,25.25) {\scalefont{5} $j$};
 
%\draw plot[domain=1:24] (\x, {24-24/(\x)+1});

 \foreach \x in {0, ..., 24}
 {
 \draw (\x,0)--(\x,24);
 \draw (0,\x)--(24,\x);
 }
 \foreach \x in {0, ..., 24}
 {
 \node at (\x,-0.5){\x};
 }

 \foreach \x in {0, ..., 24}
 {
 \node at (-0.5,\x){\x};
 }

 \foreach \x in {1,...,21}
 {
 \draw [blue, line width = 2mm] (\x+1.5,\x-0.5)--(\x+1.5,\x+0.5)--(\x+2.5,\x+0.5);
 }
\draw [blue, line width = 2mm] (1.5,0)--(1.5,0.5)--(2.5,0.5); 
\draw [blue, line width = 2mm] (23.5,21.5)--(23.5,23.5)--(24,23.5);

 \draw [blue, line width = 2mm] (0.5,0)--(0.5,0.5)--(1.5,0.5)--(1.5,12.5)--(2.5,12.5)--(2.5,16.5)--(3.5,16.5)--(3.5,18.5)--(4.5,18.5)--(4.5,19.5)--(5.5,19.5)--(5.5,20.5)--(7.5,20.5)--(7.5,21.5)--(11.5,21.5)--(11.5,22.5)--(23.5,22.5);

\node at (2,22)[blue] {\scalefont{8} \bf solid};
\node at (9,15)[blue] {\scalefont{8} \bf liquid};
\node at (18,6)[blue]  {\scalefont{8} \bf gas};

 \end{tikzpicture}
 }
 \caption{Theorem \ref{thm:regimes} describes the shapes of the homological solid, liquid, and gas regimes for every $n$. We illustrate here the case $n=24$.}
 \end{figure}
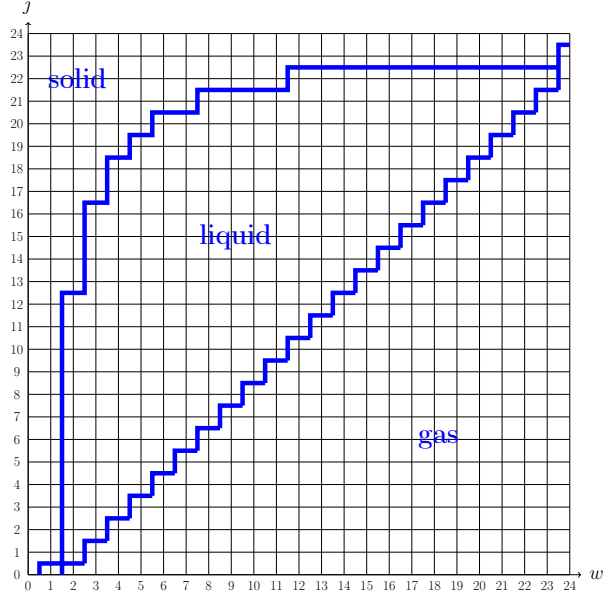

\bigskip

%{\bf \noindent Outline of rest of paper\\}

Although the statement is for every $n$ and not only asymptotically as $n \to \infty$, Theorem \ref{thm:regimes} is actually easier than Theorem \ref{thm:main} and is essentially a corollary of intermediate results. \\

The rest of the paper is organized as follows. 

In Section \ref{sec:defs}, we give definitions and notation used throughout the rest of the paper. In particular, we describe a polyhedral cell complex $\cell(n,w)$, a subcomplex of the Salvetti complex, which is homotopy equivalent to $\config(n,w)$.

In Section \ref{sec:homotopy}, we prove the homotopy equivalence of $\config(n,w)$ and $\cell(n,w)$. Parts (1) and (3) of Theorem \ref{thm:main} follow immediately from the homotopy equivalence.

In Section \ref{sec:lower}, we prove lower bounds on the Betti numbers in the liquid regime, giving one direction of part (2) of Theorem \ref{thm:main}. 

In Section \ref{sec:regimes}, we prove Theorem \ref{thm:regimes}.

In Section \ref{sec:Morse}, we describe a discrete gradient vector field on $\cell(n,w)$. This allows us to collapse $\cell(n,w)$ to a regular CW complex with far fewer cells.%, and then the number of $d$-cells is an upper bound on the Betti number $\beta_d$.  

In Section \ref{sec:upper}, we use the results from Section \ref{sec:Morse} to prove upper bounds, giving the other direction of part (2) of Theorem \ref{thm:main}.

In Section \ref{sec:comments}, we close with comments and open problems.

Finally, in an appendix by Ulrich Bauer and Kyle Parsons, we include calculation of the Betti numbers for $n \le 8$. \\

\section{Definitions and notation} \label{sec:defs}

%\subsection{The Salvetti complex}

%In this section we define two spaces that turn out to be homotopy equivalent to $\config(n, w)$, as we show in the section that follows this one.  %The first space $\cell(n, w)$ is a combinatorially-defined cell complex, and the second space $\configI(n, w+\epsilon)$ is a configuration space of vertical line segments.

We first describe a ranked poset which we denote $\poset(n)$, which is the face poset of a regular CW complex $\cell(n)$ called the \emph{Salvetti complex}. Then, afterward, we define $\cell(n, w)$ as a subcomplex of $\cell(n)$.  The Salvetti complex and related constructions have appeared implicitly or explicitly many times---see Section 3 of \cite{BZ14} for a brief review of the literature. The complex was apparently first described explicitly by Salvetti in \cite{Salvetti87}.

\begin{definition}
The poset which we denote $\poset(n)$ has as its underlying set $\symbols(n)$, defined as follows. We call the elements of $\symbols(n)$ ``symbols''. A symbol is a permutation in one-line notation $(\sigma_1 \ \sigma_2 \ \dots \ \sigma_n)$, where between each consecutive pair of elements $\sigma_i \ \sigma_{i+1}$, there can either be a bar or not. 

The bars separate the permutation into pieces that we call \emph{blocks}. The partial order on $\poset(n)$ is characterized as follows: the covers in the Hasse diagram of a symbol $\alpha$ are the symbols obtained from $\alpha$ by the  operation of removing a bar and merging the two adjacent blocks by a shuffle---the shuffle must preserve the relative order within each block.
\end{definition}

The Hasse diagram of $\poset(3)$ is illustrated in Figure \ref{fig:p3}. 
For example, $(1 \mid 3 \mid 2)$, $(3 \ 1 \mid 2)$, and $(3 \ 2 \ 1)$ are all symbols in $\symbols(3)$. Moreover, they form a chain in the poset.

There are $n-1$ positions between consecutive pairs of elements, so there are exactly $n! \, 2^{n-1}$ symbols in $\symbols(n)$.

\begin{figure}
\centering
\begin{tikzpicture}[thin]
\small
\node at (0,0) {$\boldsymbol{\bigcdot}$};
\node at (1,0) {$\boldsymbol{\bigcdot}$};
\node at (2,0) {$\boldsymbol{\bigcdot}$};
\node at (3,0) {$\boldsymbol{\bigcdot}$};
\node at (4,0) {$\boldsymbol{\bigcdot}$};
\node at (5,0) {$\boldsymbol{\bigcdot}$};

\node at (-3,1) {$\boldsymbol{\bigcdot}$};
\node at (-2,1) {$\boldsymbol{\bigcdot}$};
\node at (-1,1) {$\boldsymbol{\bigcdot}$};
\node at (0,1) {$\boldsymbol{\bigcdot}$};
\node at (1,1) {$\boldsymbol{\bigcdot}$};
\node at (2,1) {$\boldsymbol{\bigcdot}$};
\node at (3,1) {$\boldsymbol{\bigcdot}$};
\node at (4,1) {$\boldsymbol{\bigcdot}$};
\node at (5,1) {$\boldsymbol{\bigcdot}$};
\node at (6,1) {$\boldsymbol{\bigcdot}$};
\node at (7,1) {$\boldsymbol{\bigcdot}$};
\node at (8,1) {$\boldsymbol{\bigcdot}$};

\node at (0,2) {$\boldsymbol{\bigcdot}$};
\node at (1,2) {$\boldsymbol{\bigcdot}$};
\node at (2,2) {$\boldsymbol{\bigcdot}$};
\node at (3,2) {$\boldsymbol{\bigcdot}$};
\node at (4,2) {$\boldsymbol{\bigcdot}$};
\node at (5,2) {$\boldsymbol{\bigcdot}$};
\draw [opacity=0.25] (0,0) --(-3,1);
\draw [opacity=0.25] (0,0) --(-2,1);
\draw [opacity=0.25] (0,0) --(-1,1);
\draw [opacity=0.25] (0,0) --(0,1);
\draw [opacity=0.25] (1,0) -- (-1,1);
\draw [opacity=0.25] (1,0) -- (0,1);
\draw [opacity=0.25] (1,0) -- (1,1);
\draw [opacity=0.25] (1,0) -- (2,1);
\draw [opacity=0.25] (2,0) -- (1,1);
\draw [opacity=0.25] (2,0) -- (2,1);
\draw [opacity=0.25] (2,0) -- (3,1);
\draw [opacity=0.25] (2,0) -- (4,1);
\draw [opacity=0.25] (3,0) -- (3,1);
\draw [opacity=0.25] (3,0) -- (4,1);
\draw [opacity=0.25] (3,0) -- (5,1);
\draw [opacity=0.25] (3,0) -- (6,1);
\draw [opacity=0.25] (4,0) -- (5,1);
\draw [opacity=0.25] (4,0) -- (6,1);
\draw [opacity=0.25] (4,0) -- (7,1);
\draw [opacity=0.25] (4,0) -- (8,1);
\draw [opacity=0.25] (5,0) -- (7,1);
\draw [opacity=0.25] (5,0) -- (8,1);
\draw [opacity=0.25] (5,0) -- (-3,1);
\draw [opacity=0.25] (5,0) -- (-2,1);

\draw [opacity=0.25] (-3,1) -- (0,2);
\draw [opacity=0.25] (-3,1) -- (1,2);
\draw [opacity=0.25] (-3,1) -- (2,2);
\draw [opacity=0.25] (-2,1) -- (3,2);
\draw [opacity=0.25] (-2,1) -- (4,2);
\draw [opacity=0.25] (-2,1) -- (5,2);
\draw [opacity=0.25] (-1,1) -- (1,2);
\draw [opacity=0.25] (-1,1) -- (2,2);
\draw [opacity=0.25] (-1,1) -- (3,2);
\draw [opacity=0.25] (0,1) -- (0,2);
\draw [opacity=0.25] (0,1) -- (4,2);
\draw [opacity=0.25] (0,1) -- (5,2);
\draw [opacity=0.25] (1,1) -- (2,2);
\draw [opacity=0.25] (1,1) -- (3,2);
\draw [opacity=0.25] (1,1) -- (4,2);
\draw [opacity=0.25] (2,1) -- (0,2);
\draw [opacity=0.25] (2,1) -- (1,2);
\draw [opacity=0.25] (2,1) -- (5,2);
\draw [opacity=0.25] (3,1) -- (0,2);
\draw [opacity=0.25] (3,1) -- (1,2);
\draw [opacity=0.25] (3,1) -- (2,2);
\draw [opacity=0.25] (4,1) -- (3,2);
\draw [opacity=0.25] (4,1) -- (4,2);
\draw [opacity=0.25] (4,1) -- (5,2);
\draw [opacity=0.25] (5,1) -- (1,2);
\draw [opacity=0.25] (5,1) -- (2,2);
\draw [opacity=0.25] (5,1) -- (3,2);
\draw [opacity=0.25] (6,1) -- (0,2);
\draw [opacity=0.25] (6,1) -- (4,2);
\draw [opacity=0.25] (6,1) -- (5,2);
\draw [opacity=0.25] (7,1) -- (2,2);
\draw [opacity=0.25] (7,1) -- (3,2);
\draw [opacity=0.25] (7,1) -- (4,2);
\draw [opacity=0.25] (8,1) -- (0,2);
\draw [opacity=0.25] (8,1) -- (1,2);
\draw [opacity=0.25] (8,1) -- (5,2);

\node at (0,-0.35) {$\bm{ (2 | 1 | 3)}$};
\node at (1,-0.35) {$\bm{(2 | 3 | 1)}$};
\node at (2,-0.35) {$\bm{ (3 | 2 | 1)}$};
\node at (3,-0.35) {$\bm{ (3 | 1 | 2)}$};
\node at (4,-0.35) {$\bm{ (1 | 3 | 2)}$};
\node at (5,-0.35) {$\bm{ (1 | 2 | 3)}$};

\node at (0,2.35) {$\bm{ (2 1 3)}$};
\node at (1,2.35) {$\bm{ (2 3 1)}$};
\node at (2,2.35) {$\bm{ (3 2 1)}$};
\node at (3,2.35) {$\bm{ (3 1 2)}$};
\node at (4,2.35) {$\bm{ (1 3 2)}$};
\node at (5,2.35) {$\bm{ (1 2  3)}$};

\node at (-3.4,1) {$\bm{ (2 1 | 3)}$};
\node at (-2.4,1) {$\bm{ (1 2 | 3)}$};
\node at (-1.4,1) {$\bm{ (2 | 3 1)}$};
\node at (-0.4,1) {$\bm{ (2 | 1 3)}$};
\node at (0.6,1) {$\bm{ (3 2 | 1)}$};
\node at (1.6,1) {$\bm{ (2 3 | 1)}$};
\node at (2.6,1) {$\bm{ (3 | 2 1)}$};
\node at (3.6,1) {$\bm{ (3 | 1 2)}$};
\node at (4.6,1) {$\bm{ (3  1 | 2)}$};
\node at (5.6,1) {$\bm{ (1 3 | 2)}$};
\node at (6.6,1) {$\bm{ (1 | 3 2)}$};
\node at (7.6,1) {$\bm{ (1 | 2 3)}$};
\end{tikzpicture}

\caption{The Hasse diagram of $\poset(3)$. This is the face poset of the Salvetti complex for the configuration space of $3$ points in the plane.} \label{fig:p3}
\end{figure}
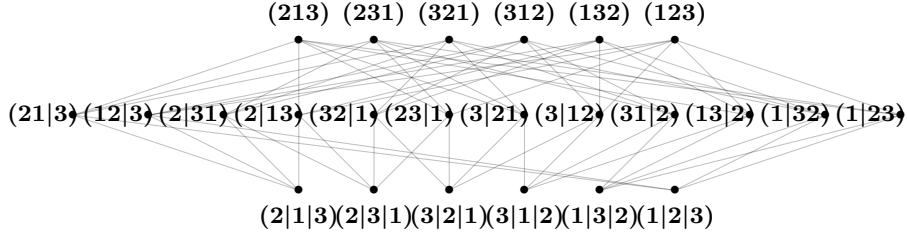

It is useful to consider ``block notation'' for a symbol. If we write
$$\alpha = (c_1 \mid c_2 \mid \dots \mid c_m),$$
it means that each $c_i$ is a block of the permutation, separated from the rest of the permutation by bars. Forgetting the order of permutation elements within a block, we may also regard a block as a subset of $[n] := \{ 1, 2, \dots, n \}$. So we may write without ambiguity such statements as ``$\sigma_k$ and $\sigma_\ell$ are in the same block''. 

It is well known that $\poset(n)$ is the face poset of a regular CW complex $\cell(n)$---see for example \cite{BZ14}, usually called the Salvetti complex.
In the more general context of complexifications of real hyperplane arrangements, it was shown in \cite{Salvetti87} that $\cell(n)$ is homotopy equivalent to the configuration space of points in the plane $\config(n,\R^2)$.

The cell complex $\cell(n)$ has $n! \binom{n-1}{i-1} = n! \binom{n-1}{n-i}$ $i$-dimensional faces, indexed by permutations with $n-i-1$ bars. If a cell is indexed by a symbol $\alpha=(c_1 \mid c_2 \mid \dots \mid c_m)$ with $m$ blocks, then the cell has dimension $j = n-m$. \\

We will be mostly concerned with certain subcomplexes of $\cell(n)$, described as follows.

\begin{definition}
For every $n , w \ge 1$, we define $\poset(n,w)$ to be the sub-poset of $\poset(n)$ where every block has width at most $w$, that is, at most $w$ elements. We note that $\poset(n,w)$ is an order ideal in $\poset(n)$.  Then since $\poset(n)$ is the face poset of $\cell(n)$, we have that $\poset(n,w)$ is the face poset of a subcomplex which we denote $\cell(n,w)$.  %Note that for $\poset(n,w)$ is a subcomplex of $\poset(n,w+1)$, and for $w \ge n$ we have $\poset(n,w) = \poset(n)$ and $\cell(n,w) = \cell(n)$. 
\end{definition}

The cell complex $\cell(3,2)$ is illustrated in Figure \ref{fig:giant}. \\

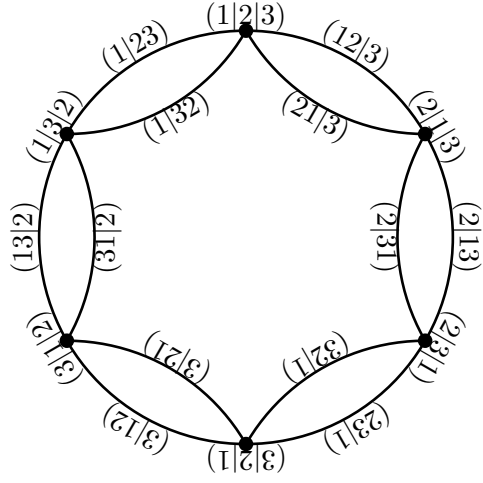
\begin{figure}
\begin{tikzpicture}[scale=2.75]
\draw [line width = 1pt, domain=150:210] plot ({1.732051+cos(\x)}, {sin(\x)});
\draw [line width = 1pt, domain=210:270] plot ({0.8660254+cos(\x)}, {1.500000+sin(\x)});
\draw [line width = 1pt, domain=270:330] plot ({-0.8660254+cos(\x)}, {1.500000+sin(\x)});
\draw [line width = 1pt, domain=-30:30] plot ({-1.732051+cos(\x)}, {sin(\x)});
\draw [line width = 1pt, domain=30:90] plot ({-0.8660254+cos(\x)}, {-1.500000+sin(\x)});
\draw [line width = 1pt, domain=90:150] plot ({0.8660254+cos(\x)}, {-1.500000+sin(\x)});
\draw [line width = 1pt] (0,0) circle (1);
\large
\node [rotate=-60] at (30:1.068) {$(2 | 1 | 3)$};
\node at (90:1.068) {$(1 | 2 | 3)$};
\node [rotate=60] at (150:1.068) {$(1 | 3 | 2)$};
\node [rotate=120]at (210:1.068) {$(3 | 1| 2)$};
\node [rotate=180]at (270:1.068) {$(3 | 2 | 1)$};
\node [rotate=240] at (330:1.068) {$(2 | 3 | 1)$};

\node[rotate=-90] at (0:1.068) {$(2 | 1  3)$};
\node [rotate=-90] at (0:0.67) {$(2 | 3  1)$};

\node [rotate=-30]at (60:1.068) {$(1 2 | 3)$};
\node [rotate=-30]at (60:0.67) {$(2 1 | 3 )$};

\node [rotate=30]at (120:1.068) {$(1 | 2  3)$};
\node [rotate=30]at (120:0.67) {$(1 | 3 2)$};

%\node at (120:1.068) {$(1 | 2  3)$};
%\node at (120:0.67) {$(1 | 3 2)$};

\node [rotate=90] at (180:1.068) {$(1  3 | 2)$};
\node [rotate=90] at (180:0.67) {$(3  1 | 2)$};

\node [rotate=150] at (240:1.068) {$(3 | 1 2)$};
\node [rotate=150] at (240:0.67) {$(3  | 2 1)$};

\node [rotate=210] at (300:1.068) {$(2 3 | 1 )$};
\node [rotate=210] at (300:0.67) {$(3  2 | 1)$};

\fill (30:1 ) circle[radius=1pt];
\fill (90:1 ) circle[radius=1pt];
\fill (150:1 ) circle[radius=1pt];
\fill (210:1 ) circle[radius=1pt];
\fill (270:1 ) circle[radius=1pt];
\fill (330:1 ) circle[radius=1pt];

\end{tikzpicture}
\caption{The cell complex $\cell(3,2)$.}
\label{fig:giant}
\end{figure}

In the remainder of the section, we define some spaces closely related to $\config(n, w)$.  We use these spaces in Section~\ref{sec:homotopy} for the proof that $\config(n, w)$ and $\cell(n, w)$ are homotopy equivalent, and in Section~\ref{sec:lower} to find lower bounds on the Betti numbers of these spaces.

Let $\config_d(n, w)$ denote the closed configuration space of $n$ disks of diameter $d$ in a strip of width $1$, so that $\config(n, w) = \config_1(n, w)$.  Rescaling gives homeomorphisms between $\config_d(n, w)$, $\config(n, w/d)$, and $\config_{d/w}(n, 1)$, so we can identify $\config(n, w)$ with $\config_{1/w}(n, 1)$.  Let $\config_0(n, 1)$ denote the union of all $\config_d(n, 1)$, i.e.,\ the configuration space of $n$ points in the strip of width $1$.

\begin{definition}
Given a symbol $\alpha \in \poset(n,w)$, we define an open set $U_{\alpha}$ in $\config_0(n, 1)$ as follows. Write $\alpha$ in block notation $\alpha = (c_1 \mid c_2 \mid \dots \mid c_m)$, and then define the open set $U_{\alpha}$ to be the set of points $(x_1, y_1, x_2, y_2, \dots, x_n, y_n) \in \R^{2n}$ such that the following conditions are met. 

\begin{itemize}
\item We have $0 < y_k < 1$ for all $y_k$.
\item Whenever $\sigma_k$ and $\sigma_{\ell}$ are in the same block and $k < \ell$, we have $y_{\sigma_k} > y_{\sigma_\ell}$.
\item Whenever $\sigma_k$ and $\sigma_{\ell}$ are in different blocks and $k < \ell$, we have $x_{\sigma_k} < x_{\sigma_\ell}$.
\item If $\sigma_k$ and $\sigma_{\ell}$ are in the same block, and $\sigma_{k'}$ and $\sigma_{\ell'}$ are in different blocks, then
$$|x_{\sigma_k} - x_{\sigma_\ell}| < |x_{\sigma_{k'}} - x_{\sigma_{\ell'}}|.$$
The indices are not assumed to be distinct---in particular it may be that $k=k'$. Intuitively, elements in the same block must cluster by $x$-coordinate.
\end{itemize}
\end{definition}

Given a symbol $\alpha$, let $w(\alpha)$ denote the largest number of elements in any block of $\alpha$.  For any natural number $w$, let $U(n, w)$ denote the union in $\config_0(n, 1)$ of all $U_\alpha$ for which $w(\alpha) \leq w$.  So $U(n, w)$ excludes exactly those configurations in $\config_0(n, 1)$ that have more than $w$ points on the same vertical line. For the sake of completeness and clarity later, we rewrite this definition as follows.

\begin{definition} \label{def:unw} We define $U(n,w)$ to the set of points $(x_1, y_1, x_2, y_2, \dots, x_n, y_n) \in \R^{2n}$ such that the following conditions are met.
\begin{itemize}
\item $(x_k,y_k) \neq (x_{\ell},y_{\ell})$ whenever $1 \le k < \ell \le n$, 
\item $0 < y_k < 1$ for every $1 \le k \le n$, and
\item no $w+1$ of the points have the same $x$-coordinate.
\end{itemize}
\end{definition}

%We define the closely related ``configuration space of vertical line segments'' as follows.
%
%\begin{definition}
%Assume that $0 < \epsilon < 1$. The configuration space of vertical line segments of unit length in the strip of width $w + \epsilon$ is defined by
%\begin{align*}
%\configI(n,w+ \epsilon) & =  \{ (x_1, y_1, \dots, x_n, y_n)  \mid  1/2 < y_i  < w  - 1/2 + \epsilon \mbox{ for every } 1 \le i \le n,\\
%& \mbox{and }\; x_i = x_j \implies |y_i -y_j| > 1  \mbox{ for every } 1 \le i < j \le n.\}
%\end{align*}
%
%\end{definition} 
%
%\medskip
%
%We prove in Section \ref{sec:homotopy} that 
%$$\config_I(n,w+ \epsilon) \hequiv \config(n,w),$$ 
%as one step in the proof of the main homotopy equivalence $\cell(n,w) \hequiv \config(n,w)$. We also use $\config_I(n,w+ \epsilon)$ in the proof of lower bounds in Section \ref{sec:lower}. It is convenient, for example, that $\config_I( n, w + \epsilon)$ is an open subset of $\R^{2n}$ and thus an open manifold.\\
%
\section{Homotopy equivalence} \label{sec:homotopy}

The main goal of this section is to prove the following theorem.

\begin{theorem} \label{thm:homoeq} For every $n, w \ge 1$, we have a homotopy equivalence $\config(n,w) \hequiv \cell(n,w)$.  Moreover, these homotopy equivalences for $w$ and $w+1$ commute up to homotopy with the inclusions $\cell(n, w) \hookrightarrow \cell(n, w+1)$ and $\config(n, w) \hookrightarrow \config(n, w+1)$.
\end{theorem}

In Section~\ref{sec:nerve} we give an overview of the parts of the proof, and in Section~\ref{sec:homotopyproof} we prove the technical lemmas needed to finish the proof.  Then, in Section \ref{sec:consequences} we list a few of the immediate consequences of the homotopy equivalence.

\subsection{Proof overview of Theorem~\ref{thm:homoeq}}
\label{sec:nerve}

Our strategy is to use the nerve theorem.  We briefly review some of the terminology and ideas of nerve theory.  We say that an open cover $\mathcal{U} =(U_k)_{k \in \mathcal{K}}$ of a topological space $X$ is \emph{good} if every intersection of elements of $\mathcal{U}$ is either empty or contractible.  The \emph{nerve} $N(\mathcal{U})$ of $\mathcal{U}$ is the simplicial complex built by taking a vertex for each open set $U_k$ and a simplex for every collection of open sets with nonempty intersection.  For our purposes here, the indexing set $\mathcal{K}$ will always be finite, so the nerve is always finite-dimensional.

The nerve theorem says that the nerve $N(\mathcal{U})$ is homotopy equivalent to the original space $X$.  In particular, let $\{\phi_k\}_{k \in \mathcal{K}}$ be a partition of unity subordinate to $\mathcal{U}$, and let $\{v_k\}_{k \in \mathcal{K}}$ be the vertices of the nerve.  Then the map $r \co X \rightarrow N(\mathcal{U})$ defined by
\[r(x) = \sum_{k \in \mathcal{K}} \phi_k(x)v_k\]
is a homotopy equivalence.

To prove that the homotopy equivalences in Theorem~\ref{thm:homoeq} commute with the inclusions into wider strips, we use the following ``extended'' nerve theorem, which appears in Section 4 of Bauer, Edelsbrunner, Jablonski, and Mrozek's \cite{BEJM17}. (See also Proposition 4.2 in Ferry, Mischaikow, and Nanda's \cite{FMN14}.)

\begin{theorem}[Extended nerve theorem \cite{BEJM17}] \label{thm:extendnerve} Suppose that $\mathcal{U}=(U_k)_{k \in \mathcal{K}}$ is a good open cover of a topological space $X$, and that $\mathcal{V}=(V_{\ell})_{\ell \in \mathcal{L}}$ is a good open cover of a topological space $Y$. Suppose that $f\co X \to Y$ is continuous, and that $g\co \mathcal{K} \to \mathcal{L}$ is such that $f(U_k) \subseteq V_{g(k)}$ for every $k \in \mathcal{K}$. Let $\overline{g} \co N(\mathcal{U}) \to N(\mathcal{V})$ be the linear simplicial map induced by $g$. Then the following diagram commutes, up to homotopy.

\centering

\begin{tikzcd}
X \arrow[r, "f"] \arrow[d, "r"]
& Y \arrow[d, "r"] \\
N(\mathcal{U}) \arrow[r, "\overline{g}"]
& N(\mathcal{V})
\end{tikzcd}
\end{theorem} 

The other two pieces we need for the proof of Theorem~\ref{thm:homoeq} are the following two theorems, which we prove in Section~\ref{sec:homotopyproof}.

\begin{theorem} \label{thm:deformretract}
$U(n, w)$ deformation retracts to $\config_{1/w}(n, 1)$.
\end{theorem}

\begin{theorem} \label{thm:barycentric}
The cover $\{U_\alpha\}_{w(\alpha) \leq w}$ of $U(n, w)$ is good, and its nerve $N(n, w)$ is the barycentric subdivision of the regular cell complex $\cell(n, w)$.
\end{theorem}

Assuming these two theorems, we can finish the proof of Theorem~\ref{thm:homoeq}.

\begin{proof}[Proof of Theorem~\ref{thm:homoeq}]
The homotopy equivalence between $\config(n, w)$ and $\cell(n, w)$ is a composition of homotopy equivalences
\[\config(n, w) \rightarrow \config_{1/w}(n, 1) \rightarrow U(n, w) \rightarrow N(n, w) \rightarrow \cell(n, w).\]
Specifically, let $\phi \co \config(n, w) \rightarrow \config_{1/w}(n, w)$ be the rescaling map, which is a homeomorphism.  The inclusion map $i \co \config_{1/w}(n, w) \hookrightarrow U(n, w)$ is a homotopy equivalence because its homotopy inverse is the deformation retraction from Theorem~\ref{thm:deformretract}.  The map $r \co U(n, w) \rightarrow N(n, w)$ is the homotopy equivalence given by the nerve theorem.  And, applying Theorem~\ref{thm:barycentric} we let $b \co N(n, w) \rightarrow \cell(n, w)$ be the homeomorphism that undoes the barycentric subdivision.

We check that each of these maps commutes up to homotopy with the inclusions resulting from mapping each space to the corresponding space with $w$ replaced by $w+1$.  The diagram

{\centering

\begin{tikzcd}
\config(n,w) \arrow[r, "i"] \arrow[d, "\phi"] & \config(n,w+1) \arrow[d, "\phi"] \\
\config_{1/w}(n,1) \arrow[r, "i"] & \config_{1/(w+1)}(n,1)
\end{tikzcd}

}
\noindent commutes because inclusion commutes with rescaling.  The diagram \\

{\centering

\begin{tikzcd}
\config_{1/w}(n,1) \arrow[r, "i"] \arrow[d, "i"] & \config_{1/(w+1)}(n,1) \arrow[d, "i"] \\
 \arrow[r, "i"]
U(n,w) & U(n,w+1) 
\end{tikzcd}

}
\noindent commutes because all of the maps are inclusions. The diagram\\

{\centering

\begin{tikzcd}
U(n,w) \arrow[r, "i"] \arrow[d, "r"] & U(n,w+1) \arrow[d, "r"] \\
 \arrow[r, "i"]
N(n,w) & N(n,w+1)
\end{tikzcd}

}
\noindent commutes by the extended nerve theorem, Theorem \ref{thm:extendnerve}.  And, the diagram\\

{\centering

\begin{tikzcd}
N(n,w) \arrow[r, "i"] \arrow[d, "b"] & N(n,w+1) \arrow[d, "b"] \\
 \arrow[r, "i"]
\cell(n,w) & \cell(n,w+1)
\end{tikzcd}

}
\noindent commutes because barycentric subdivision is functorial.
\end{proof}

\subsection{Technical lemmas for Theorem \ref{thm:homoeq}} 
\label{sec:homotopyproof}

First we prove Theorem~\ref{thm:deformretract}, which says that $U(n, w)$ deformation retracts to $\config_{1/w}(n, 1)$.  Let $\tau \co \config_0(n, 1) \rightarrow (0, \infty)$ denote the function defined as follows; the paper~\cite{BBK14} calls this the \emph{tautological function}.  For any configuration $p$, we set $\tau(p)$ to be the maximum value $d$ such that $p \in \config_d(n, 1)$.  Intuitively, we take the points in the configuration $p$, and consider disks of growing radius with those points as centers.  When the disks first become tangent to each other or to the lines $y = 0$ or $y=1$, the diameter of those disks is $\tau(p)$.

Because $\config_{1/w}(n, 1)$ is the subset of $U(n, w)$ where $\tau \geq 1/w$, our strategy for the deformation retraction is to flow along a vector field that increases $\tau$.  In the following lemma, we construct such a vector field for each $U_\alpha$, and then we combine the vector fields for various $\alpha$ together.

\begin{lemma} \label{lem-valpha}
For each set $U_\alpha$, there is a continuous vector field $\{v_\alpha(p)\}_{p \in U_\alpha}$, with the property that if $\tau(p) < 1/w(\alpha)$, then $\tau$ is increasing at rate at least $2\sqrt{1/w(\alpha) - \tau(p)}$ in the direction $v_\alpha(p)$, that is, $D\tau_p(v_\alpha(p)) \geq 2\sqrt{1/w(\alpha) - \tau(p)}$.
\end{lemma}

(Although $\tau$ is not a smooth function, it turns out that its directional derivative in every direction is well-defined, so the expression $D\tau_p(v_\alpha(p))$ makes sense.)  Before proving this technical lemma, we show how it gives the deformation retraction needed for Theorem~\ref{thm:deformretract}.

\begin{proof}[Proof that Lemma~\ref{lem-valpha} implies Theorem~\ref{thm:deformretract}]
Let $w$ be a natural number.  To define the deformation retraction from $U(n, w)$ to $\config_{1/w}(n, 1)$, we combine the vector fields $v_\alpha(p)$ from Lemma~\ref{lem-valpha} as follows to produce a vector field $v(p)$ on $U(n, w) \setminus \config_{1/w}(n, 1)$.  Let $\phi_\alpha$ be a partition of unity subordinate to the covering $\{U_\alpha\}_{w(\alpha) \leq w}$ of $U(n, w)$.  That is, for any $p \in U(n, w)$, each $\phi_\alpha(p)$ is between $0$ and $1$, with $\phi_\alpha(p) = 0$ if $p \not\in U_\alpha$, and $\sum_{\alpha} \phi_\alpha(p) = 1$.  For each $p \in U(n, w) \setminus \config_{1/w}(n, 1)$, we define
\[v(p) = \sum_\alpha \phi_\alpha(p) v_\alpha(p).\]

We claim that for each $p$, if $\tau(p) < 1/w$ then we have $D\tau_p(v(p)) \geq 2\sqrt{1/w - \tau(p)}$; that is, $\tau(p)$ increases at rate at least $2\sqrt{1/w - \tau(p)}$ in the direction $v(p)$.  We already have that each vector field $v_\alpha(p)$ increases $\tau(p)$ at this rate:
\[D\tau_p(v_\alpha(p)) \geq  2\sqrt{1/w(\alpha) - \tau(p)} \geq 2\sqrt{1/w - \tau(p)}.\]
  For the convex combination $v_\alpha(p)$, we look at the configuration $p$ of disks of diameter $\tau(p)$, and find all the places where two disks are tangent or where a disk is tangent to the boundary of the strip.  For each of these tangencies, there is a linear functional on the tangent space at $p$, measuring how the distance between those two disks or between that disk and the boundary changes as $p$ varies.  The set of vectors at $p$ that increase $\tau$ is the intersection of half-spaces, one for each of the tangencies.  And, the set of vectors at $p$ that increase $\tau$ at rate at least $2\sqrt{1/w - \tau(p)}$ is the intersection of half-spaces, one for each of the tangencies, cut out by planes that are parallel to the corresponding planes for finding the set of vectors at $p$ that increase $\tau$.  Thus, this latter set is convex.  The vector $v(p) = \sum_\alpha \phi_\alpha(p) v_\alpha(p)$ is a convex combination of vectors $v_\alpha$ that lie in the convex set, so $v(p)$ is also in the set.

The deformation retraction is given by flowing along $v(p)$ until we reach the set $\config_{1/w}(n, 1)$.  To see how fast it finishes, let $p(t)$ be one of the trajectories, and let $\delta(t) = 1/w - \tau(p(t))$.  Then we have that $\delta(t)$ is decreasing at a rate
\[\abs{\delta'(t)} \geq 2\sqrt{\delta(t)},\]
so we have
\[\frac{\delta'(t)}{2\sqrt{\delta(t)}} \leq -1,\]
or
\[\frac{d}{dt}\sqrt{\delta(t)} \leq -1.\]
Thus the quantity $\sqrt{\delta(t)}$ decreases from at most $\sqrt{1/w}$ to $0$ at rate at least $1$, and so the deformation retraction finishes in time at most $\sqrt{1/w} \leq 1$.
\end{proof}

The bulk of Theorem~\ref{thm:deformretract} is in constructing the vector field $v_\alpha(p)$ to prove Lemma~\ref{lem-valpha}.

\begin{proof}[Proof of Lemma~\ref{lem-valpha}]
Given any $\alpha$, we construct $v_\alpha(p)$ as 
\[v_\alpha(p) = \lambda(p)(y(p) - p) + (x(p) - p),\]
where $x \co U_\alpha \rightarrow U_\alpha$ and $y \co U_\alpha \rightarrow U_\alpha$ specify configurations, and $\lambda \co U_\alpha \rightarrow \mathbb{R}_{\geq 0}$ specifies a nonnegative scaling.  The configuration $y(p)$ differs from $p$ by moving the points vertically, and the configuration $x(p)$ differs from $p$ by moving the points horizontally such that points in the same block of $\alpha$ keep their relationships but the horizontal space between blocks may increase.  We first construct $y(p)$, then construct $\lambda(p)$, and then construct $x(p)$ so that the resulting vector has the desired properties.  

To analyze whether the vector $v_\alpha(p)$ increases $\tau(p)$ quickly enough, we observe that $\tau(p)$ is the minimum of all the distances between pairs of points in $p$ and twice all the distances between the points and the boundary of the strip.  We refer to these distances as the \emph{measurements}, and refer to each measurement less than $1/w(\alpha)$ as \emph{short}.  It suffices to construct $v_\alpha(p)$ such that for each short measurement $m(p)$, moving along the vector $v_\alpha(p)$ increases $m(p)$ at rate at least $2\sqrt{1/w(\alpha) - m(p)}$.

We start by constructing $y(p)$ in such a way that for each short measurement within any block, moving from $p$ to $y(p)$ increases that short measurement.  Let $y(p) \in U_\alpha$ be a configuration in which the points have the same $x$--coordinates as in $p$, but the $y$--coordinates are evenly spaced within each block, in the following way.  If a block has $k$ elements, then the $y$--coordinates of those points in $y(p)$ are $1 - \frac{1}{2}k, 1 - \frac{3}{2}k, \ldots, \frac{3}{2}k, \frac{1}{2}k$, so that the intervals of size $\frac{1}{k}$ around these values exactly tile the interval from $0$ to $1$.  In the new configuration $y(p)$, each distance between points in the same block is at least $1/w(\alpha)$, and twice the distance from each point to the boundary of the strip is at least $1/w(\alpha)$.

We still need to check that the vector at $p$ given by $y(p) - p$ infinitesimally increases each short measurement within any block.  Consider any short measurement in $p$ given by distance from a point to the boundary of the strip.  In $y(p)$, this measurement is no longer short, so the vector $y(p) - p$ must move the relevant point away from the boundary of the strip, increasing that measurement.  

Next, consider any short measurement in $p$ given by distance between two points in the same block.  Let $a$ and $b$ be the points in $p$, and let $a'$ and $b'$ be the corresponding points in $y(p)$.  To determine whether the vector $y(p) - p$ increases this measurement, we need to look at the triangle formed by vectors $b-a$ and $b'-a'$.  The measurement increases if and only if the angle at $b-a$ is obtuse; that is, if we have the inequality of inner products $\langle b'-a', b-a\rangle > \langle b-a, b-a \rangle$.  Because our two points are in the same block, the $x$--coordinates of $b-a$ and $b'-a'$ are the same, while the $y$--coordinate of $b' - a'$ has the same sign as that of $b-a$ but has greater magnitude because our measurement is no longer short in $y(p)$.  Thus, we do have the desired inequality.

Because the vector $y(p) - p$ infinitesimally increases each short measurement within any block, we can choose the scaling $\lambda(p)$ such that if $m(p)$ is a short measurement within any block, then $\lambda(p)(y(p) - p)$ increases $m(p)$ at a rate of at least $2\sqrt{1/w(\alpha) - m(p)}$.  We choose the minimum possible scaling $\lambda(p)$ with this property.  In particular, if $p$ has no short measurements, or if every short measurement in $p$ is a distance between points in different blocks, then $\lambda(p) = 0$.

Next we choose the configuration $x(p)$, which differs from $p$ by sliding some pairs of consecutive blocks away from each other.  The vector $x(p) - p$ does not change any measurements within any blocks.  Thus, our task is to choose $x(p)$ such that the resulting vector $v_\alpha(p) = \lambda(p)(y(p) - p) + (x(p) - p)$ increases each short measurement between blocks by the desired amount; the measurements within blocks are already taken care of.  We choose $x(p)$ such that the left-most block of $p$ does not move.  From this assumption, the configuration $x(p)$ is determined by specifying the amount of horizontal space between each pair of consecutive blocks.

Consider two consecutive blocks.  If none of the distances between one point in the first block and another point in the second block are short measurements, then we leave the distance between those two blocks the same.  Otherwise, consider such a short measurement.  Let $a$ and $b$ be the points in $p$ that give the measurement, with $a$ in the left block and $b$ in the right block.  Suppose our vector $v_\alpha(p)$ moves our points toward $a'$ and $b'$; that is, for some small $\varepsilon > 0$ such that $p + \varepsilon \cdot v_\alpha(p) \in U_\alpha$, let $a'$ and $b'$ be the corresponding points in the configuration $p + \varepsilon \cdot v_\alpha(p)$.  (Note that whether $p + \varepsilon \cdot v_\alpha(p)$ is in $U_\alpha$ depends on $\lambda(p)$ and $y(p)$ but not on $x(p)$.)  Then, as above, the measurement increases if and only we have $\langle b' - a', b-a \rangle > \langle b-a, b-a\rangle$, or equivalently if $\langle (b' - a') - (b-a), b-a \rangle > 0$.  In fact, the rate that $\tau$ increases in the direction $v_\alpha(p)$ is $1/\varepsilon$ times the length of the projection of $(b' - a') - (b-a)$ onto the direction $b-a$.  We claim that we can choose $x(p)$ such that this projection is sufficiently long.

Let $a = (a_1, a_2)$, $b = (b_1, b_2)$, $a' = (a'_1, a'_2)$, and $b' = (b'_1, b'_2)$.  Then $a_2$, $b_2$, $a'_2$, and $b'_2$ are determined by our choice of $\lambda(p)$, $y(p)$, and $\varepsilon$.  We have that $b_1 - a_1$ is positive (because $a$ is to the left of $b$), and that $(b'_1 - a'_1) - (b_1 - a_1)$ is equal to $\varepsilon$ times the amount of additional space in $x(p)$ between the two blocks, compared to the space in $p$.  By choosing this amount of additional space to be large, we may cause the quantity $((b'_1 - a'_1) - (b_1 - a_1))(b_1 - a_1)$ to be arbitrarily large, while keeping the quantity $((b'_2 - a'_2) - (b_2 - a_2))(b_2 - a_2)$ the same, thus making the inner product $\langle (b' - a') - (b-a), b-a \rangle$ as positive as we want.  Thus, there is some choice of how much more space $x(p)$ should have than $p$ between the two blocks, in order to have the property that for each short measurement $m(p)$ between the two blocks, the vector $v_{\alpha}(p)$ increases this measurement at rate at least $2\sqrt{1/w(\alpha) - m(p)}$; we choose the least possible such amount of additional space.

Repeating this computation for each pair of consecutive blocks gives $x(p)$ and thus completes the construction of $v_\alpha(p)$.  The selection of $\lambda(p)$ and $y(p)$ guarantees that $v_\alpha(p)$ increases each short measurement $m(p)$ within a block at rate at least $2\sqrt{1/w(\alpha) - m(p)}$, and the selection of $x(p)$ guarantees that $v_\alpha(p)$ increases each short measurement $m(p)$ between two blocks at rate at least $2\sqrt{1/w(\alpha) - m(p)}$.  Thus, $v_\alpha(p)$ increases the function $\tau(p)$, equal to the minimum of all these measurements, at rate at least $2\sqrt{1/w(\alpha) - \tau(p)}$.
\end{proof}

This completes the proof of Theorem~\ref{thm:deformretract}, the deformation retraction from $U(n, w)$ to $\config_{1/w}(n, 1)$.  Next we prove Theorem~\ref{thm:barycentric}, which implies using the nerve theorem that $U(n, w)$ is homotopy equivalent to $\cell(n, w)$.  The main thing to check is that the intersection 
$$ U_{\alpha_1} \cap U_{\alpha_2} \cap \dots \cap U_{ \alpha_k} $$
is nonempty if and only if the symbols
$$ \{ \alpha_1, \alpha_2, \dots, \alpha_k \}$$ form a chain in $\poset(n,w)$; thus, a simplex in the nerve $N(n, w)$ corresponds to a chain of incident cells in $\cell(n, w)$.

\begin{proof}[Proof of Theorem~\ref{thm:barycentric}]

Every $U_{\alpha}$ is a convex open subset of $\R^{2n}$, so every $U_{\alpha}$ is contractible and since the intersection of convex sets is convex, every nonempty intersection is contractible.  Thus, the sets $U_\alpha$ form a good cover of $U(n, w)$.

Given a point $p \in U(n,w)$, we first describe an algorithm for finding $\mathcal A_p$, the set of symbols $\alpha \in \config(n,w)$ such that $p \in U_{\alpha}$. Along the way, we will see that $\mathcal A_p$ is a chain.

We first define the poset of ordered partitions $\partition(n)$.  An element of $\partition(n)$ is an ordered sequence $(S_1,S_2, \ldots)$ of non-empty subsets of $[n]$ such that the subsets $S_j$ are pairwise disjoint, and their union is all of $[n]$.  %This is very similar to an element of $\symbols(n)$, where we forget the order within a block (replacing an ordered subset of $[n]$ by a subset). %So we can use similar notation $(S_1 \mid S_2 \mid \dots \mid S_m)$.

The partial order on $\partition(n)$ is characterized as follows: the covers of an ordered partition $\pi$ are the ordered partitions obtained from $\pi$ by the operation of replacing two adjacent subsets by their union at the same place in the order. We remark that $\partition(n)$ is somewhat similar to $\poset(n)$, but in $\partition(n)$ we forget the order of the elements within a block. 

Now let a point $$p = (x_1, y_1, x_2, y_2, \dots, x_n, y_n) \in U(n,w)$$
be given.  We find the set of $U_\alpha$ containing $p$ in the following way.\\

\noindent {\bf Step 1} produces a chain  $\pi_1, \pi_2, \dots, \pi_m$ in the poset $\partition(n)$.  This step uses the $x$ coordinates of $p$ but not the $y$ coordinates.  We say that $x_k$ and $x_\ell$ are \emph{consecutive $x$ values of $p$} if $x_k<x_\ell$ and there is no $k'$ for which $x_k < x_{k'} < x_\ell$.  
%
%For every real number $\rho \geq 0$, there is a unique partition $$\pi(\rho) = (S_1(\rho) \mid S_2(\rho) \mid \dots \mid S_{m(\rho)})$$
% such that if $x_i$ and $x_j$ are consecutive $x$ values and $i \in S_k(\rho)$, $j\in S_\ell(\rho)$, then 
%\begin{itemize} 
%\item if $x_k-x_{\ell} \leq \rho$ then $k=\ell$, i.e.\ $k$ and $l$ lie in the same set of the partition $\pi(\rho)$, and 
%\item if $x_k-x_l > \rho$ then $k+1=\ell$, i.e.\ $k$ and $l$ lie in consecutive sets of the partition $\pi(\rho)$.
%\end{itemize} 

%The existence and uniqueness of $\pi(\rho)$ is easily proved by induction on the $x$ values.  Here's a topological construction of it.  
For $i = 1, 2, \dots, n$ and $\rho \ge 0$, define $K_i(\rho):= [x_i-\rho /2,x_ i+ \rho /2]$, i.e.,\ the closed interval of length $\rho$ centered on $x_i$. Let $K(\rho) := \bigcup K_i(\rho)$ be the union of all the intervals $K_i(\rho)$.  Every path component of $K(\rho)$ is a union of finitely many closed intervals and is connected, hence is a closed interval itself. We cluster the integers $[n]$ according to which path connected component of $K$ they lie in. In other words, for $i \in [n]$ we say that $i \in S_k(\rho)$ if $x_i$ is in the $k$th connected component of $K(\rho)$, counting left to right.

When $\rho = 0$, $k$ and $\ell$ lie in the same cluster if and only if $x_k=x_{\ell}$.  When $\rho$ is sufficiently large, there is only one cluster and $S_1(\rho) = [n]$.  In general, the ordered partition $\pi(\rho)$ changes only at certain values of $\rho$, namely the differences of consecutive $x$ values.  So as $\rho$ increases, we get a finite sequence of distinct ordered partitions $\pi_1, \pi_2, \ldots$ This sequence is the desired chain in $\partition(n)$. \\

\noindent {\bf Step 2} produces a chain $\tilde{\pi}_1, \tilde{\pi}_2, \dots, \tilde{\pi}_{m'}$ in $\poset(n,w)$. Here $m' \le m$ and for every $1 \le i \le m'$, the symbol $\tilde{\pi}_i \in \poset(n,w)$ is a ``lift'' of the ordered partition $\pi_i \in \partition(n)$. This step uses the $y$ coordinates of $p$ but not the $x$ coordinates.

Given a partition $\pi_i=(S_1, S_2, \dots)$ produced in step one, order the elements within each subset $S_k$ to produce $\tilde{\pi}_i \in \poset(n, w)$, in such a way that if $\sigma_\ell$ and $\sigma_{\ell '}$ are elements of $S_k$ with $\sigma_{\ell}$ before $\sigma_{\ell'}$ in the ordering (that is, $\ell < \ell'$), then $y_{\sigma_{\ell}} > y_{\sigma_{\ell'}} + 1$.  If for some $S_k$ this can not be done, then discard $\pi_i$ from the chain and exclude it from further consideration. We can also discard, then, $\pi_k$ for any $k > i$; any such ordered partition is just made by merging elements of $\pi_i$, so it is still impossible to order within a part by $y$ coordinate.

%Since a symbol in $\symbols(n)$ is just an ordered partition $(S_1, S_2,\ldots)$ of $[n]$ together with an ordering of the elements of each subset $S_j$, we now have a list of symbols in $\symbols(n)$.  Using the fact that $\pi_1, \pi_2,\ldots$ is a chain in $\partition(n)$, we see that our list of symbols is a chain in $\poset(n)$.  It actually lies in the sub-poset $\poset(n,w)$ because the $y_i$ values of $p$ are restricted to lie between $1/2$ and $w+\epsilon - 1/2$. \\
%
%\noindent {\bf Step 3}  extracts a subchain using the $x_i$ values.   From the list of symbols $\alpha$  coming from Step 2, throw out those that don't satisfy $p\in \mathcal A_\alpha$.  The result will still be a chain, since a subsequence of a chain is a chain.  

It is immediate that the chain constructed from steps one and two is the desired $\mathcal{A}_p$, from the definition of $U_{\alpha}$.\\

%The algorithm always produces a nonempty chain, as follows. The point $p$ is in  $\configI(n,w+\epsilon)$ by assumption, so $\rho=0$ in step one gives a partition $\pi_1 \in \partition(n)$. Moreover, this partition lifts to a partition $\tilde{\pi}_1 \in \poset(n,w)$; one can order the elements within a block by $y$-coordinate by definition of $\configI(n,w+\epsilon)$. We conclude that the set $\mathcal{A}_p$ is always nonempty, and therefore the set $U_{\alpha}$ form an open cover of $\configI(n,w+\epsilon)$.

%Every $U_{\alpha}$ is a convex subset of $\R^{2n}$, so every $U_{\alpha}$ is contractible and since the intersection of convex sets is convex, every nonempty intersection is contractible.

Finally, we check that an intersection 
$$ U_{\alpha_1} \cap U_{\alpha_2} \cap \dots \cap U_{ \alpha_k} $$
is nonempty if and only if the symbols
$$ \{ \alpha_1, \alpha_2, \dots, \alpha_k \}$$ form a chain in $\poset(n, w)$. 

First of all, if the intersection is nonempty then we apply the algorithm to any point $p$ in the intersection.  The set of $U_{\alpha}$ containing $p$ gives a chain that includes all of $\alpha_1, \ldots, \alpha_k$, and any sub-poset of a chain is a chain, so $\alpha_1, \ldots, \alpha_k$ must form a chain.

% Now suppose we have a chain $\alpha_1 < \alpha_2 < \dots < \alpha_m$ in $\poset(n,w)$. We produce a point $p$ such that $p \in U_{\alpha_i}$ for every $i$. We may assume without loss of generality that the chain is maximal.

% None of the blocks in any symbol have width more than $w$. So we can choose $y$ coordinates based on the symbol $\alpha_m$ as follows. 
% If $\sigma_k$ and $\sigma_l$ are in the same block 
% and $k < l$ then choose $y_k > y_l + 1$, then there is enough room vertically in $\configI(n,w+\epsilon)$ to do this.

% For $x$-coordinates, we first look at the symbol $\alpha_1$, which by maximality of the chain has all blocks of size one. Our first restriction is that we choose $x$-coordinates of $p$ such that if $\sigma_k < \sigma_l$  in $\alpha_1$ then $x_k < x_l$.

% Finally, we cluster $x$-coordinates so that as one ascends the chain, the blocks merge in the correct order. For example, let $0 < \lambda_1 < \lambda_2 < \dots < \lambda_m$ be any increasing sequence of real numbers. If symbols $\alpha_i$ and $\alpha_{i+1}$ differ by merging blocks $c_{k_i}$ and $c_{k_{i}+1}$, our only other restriction is simply to make sure that
% $$\max \left\{ |x_a- x_b |  \ : \ \sigma_a \in  c_{k_i}, \, \sigma_b \in c_{k_{i+1}}  \right\} = \lambda_i,$$
% and
% $$\max \left\{ |x_a- x_b |  \ : \ \sigma_a \mbox{ and } \sigma_b \mbox{ are in the same block in } c_{k_i}  \right\} < \lambda_i.$$

Now suppose we have a chain $\alpha_1 < \alpha_2 < \dots < \alpha_m$ in $\poset(n, w)$.  We produce a point $p$ such that $p \in U_{\alpha_i}$ for every $i$.  First we choose the $y$-coordinates.  None of the blocks in any symbol have width more than $w$, so there is enough room vertically in $U(n, w)$ to ensure that if $\sigma_k$ and $\sigma_\ell$ are in the same block of $\alpha_m$ with $k < \ell$, then we can choose $y_{\sigma_k} > y_{\sigma_\ell} + 1$.

To choose the $x$-coordinates, we start by assuming without loss of generality that the chain is maximal, so getting to each symbol $\alpha_i$ from the previous symbol $\alpha_{i-1}$ corresponds to merging two consecutive blocks.  We start with $\alpha_1$ and add restrictions on the $x$-coordinates one step at a time, so that on the $i$th step we will have fixed the differences between $x$-coordinates within each block of $\alpha_i$, but we think of the separate blocks sliding freely from side to side.  After all the steps, we will have specified the configuration up to horizontal translation.

More precisely, at step $1$ we require that if $k < \ell$---that is, if $\sigma_k$ appears before $\sigma_\ell$ in $\alpha_1$---then $x_{\sigma_k} < x_{\sigma_\ell}$, with no other restrictions.  Any such configuration is in $U_{\alpha_1}$.  At step $2$, two consecutive elements in $\alpha_1$ together form a block of width $2$, and we introduce the restriction that their $x$-coordinates have difference $9=3^2$.  Then, continuing in the same way, at step $i$ two consecutive blocks $c_{k_i}$ and $c_{k_{i+1}}$ in $\alpha_{i-1}$ merge to give $\alpha_i$.  We introduce the restriction that the difference in $x$-coordinates between the first element of $c_{k_i}$ and the last element of $c_{k_{i+1}}$---where ``first'' and ``last'' are still taken in terms of the first symbol $\alpha_1$---should be $3^i$.  

Any configuration that satisfies the restrictions up through step $i$ and also leaves horizontal gaps larger than $3^i$ between the blocks of $\alpha_i$ is in $U_{\alpha_i}$.  Note that step $i$ sets the gap between blocks $c_{k_i}$ and $c_{k_{i+1}}$ of $\alpha_{i-1}$ to be more than $3^{i-1}$, which is what we need in order for the final configuration to be in $U_{\alpha_{i-1}}$.  This is because, if we use the word ``width'' here to mean the range of $x$-coordinates, the widths of $c_{k_i}$ and $c_{k_{i+1}}$ have been set to be distinct powers of $3$ less than $3^i$, or to be $0$ if the block has only one element.  Thus, the gap has size at most $3^i - (3^{i-1} + 3^{i-2}) > 3^{i-1}$.

In the final step, step $n$ means merging two blocks to get $\alpha_{n}$ which has only one block, and at step $n$ we set the difference $x_{\sigma_n} - x_{\sigma_1}$ to be $3^n$ (here, the numbering $\sigma_1, \ldots, \sigma_n$ is still taken in terms of the first symbol $\alpha_1$).  At this stage we have specified the configuration $p$ up to translation, and it is in $U_{\alpha_1} \cap \dots \cap U_{\alpha_n}$.

\medskip

The barycentric subdivision of $\cell(n, w)$ has one vertex for every cell in $\cell(n, w)$, and one simplex for every chain of incident cells in $\cell(n, w)$.  The nerve $N(n, w)$ has one vertex for each $U_\alpha$, and thus for each cell in $\cell(n, w)$.  And, we have just shown that every set of $U_\alpha$ with nonempty intersection---corresponding to a simplex in $N(n, w)$---corresponds to a chain of incident cells in $\cell(n, w)$, and vice versa.  Thus, $N(n, w)$ is equal to the barycentric subdivision of $\cell(n, w)$.
\end{proof}

\subsection{Consequences of the homotopy equivalence} \label{sec:consequences}

One immediate consequence of the homotopy equivalence is Part (1) of Theorem \ref{thm:main}, i.e.,\ given a sufficiently wide strip we have an isomorphism on homology with the configuration space of points in the plane.

%\begin{theorem} \label{thm:stable}
%If $j \le w-2$ and $i$ is the inclusion map $$i: \config(n,w) \hookrightarrow \config(n, \R^2),$$ then the induced map on homology $$i_*: H_j ( \config(n,w) ) \to H_j( \config(n,\R^2) )$$ is an isomorphism.
%\end{theorem}

\begin{proof}[Proof of Part (1) of Theorem \ref{thm:main}] We show that if $w \ge j+2$, then we have an isomorphism between $H_j[\config(n, w)]$ and $H_j[\config(n, \R^2)]$. Moreover, this isomorphism is induced by the inclusion map $i: \config(n,w) \to \config(n, \R^2)$.

Every cell in $\cell(n)$ but not in $\cell(n,w)$ is indexed by a symbol with at least one block of width at least $w+1$. Hence every such cell has dimension at least $w$. 
Therefore, the subcomplex $\cell(n, w)$ of $\cell(n)$ has the same $(w-1)$--skeleton as $\cell(n)$, so the inclusion $i \co \cell(n, w) \hookrightarrow \cell(n)$ induces an isomorphism of homology in degrees $\leq w-2$.

By Theorem~\ref{thm:homoeq}, the homotopy equivalences 
\[\cell(n, w) \rightarrow U(n, w) \rightarrow \config(n, w)\]
and
\[\cell(n) = \cell(n, n) \rightarrow U(n, n) = \config(n, \mathbb{R}^2)\]
commute up to homotopy with the inclusions $\cell(n, w) \hookrightarrow \cell(n, n)$, $U(n, w) \hookrightarrow U(n, n)$, and $\config(n, w) \hookrightarrow \config(n, \mathbb{R}^2)$, and so the inclusion $i \co \config(n, w) \hookrightarrow \config(n, \mathbb{R}^2)$ also induces an isomorphism on homology in degrees $\leq w-2$.

% Therefore, the $(w-1)$-skeleton of $\cell(n,w)$ is the same as the $(w-1)$-skeleton of $\cell(n)$, which is homotopy equivalent to $\config(n, \R^2)$. The homology in degrees $\le w-2$ only depends on the $(w-1)$-skeleton, by cellular approximation.

% Now consider the following diagram.\\

% {

% \centering
% \begin{tikzcd}
% H_j [\config(n,w)] \arrow[r, "i_*"] \arrow[d, "h_*"] & H_j[ \config(n,w+1)] \arrow[d, "h_*"] \\
% H_j[\cell(n,w)] \arrow[r, "i_*"] & H_j[\cell(n,w+1)]
% \end{tikzcd}

% }

% By Theorem \ref{thm:diagram}, the diagram commutes. By Theorem \ref{thm:homoeq}, the maps $h_*$ are isomorphisms on homology. By the preceding cellular approximation argument, the bottom map $i_*: H_j[\cell(n,w)] \to H_j[\cell(n,w+1)]$ is an isomorphism. Then the top map $i_*: H_j[\config(n,w)] \to H_j[\config(n,w+1)]$ is also an isomorphism.

\end{proof}

Another consequence of the homotopy equivalence is that $\cell(n,2)$ is an Eilenberg--MacLane space.  This cell complex is a cube complex because for each dimension $j$, each $j$--cell is labeled by a symbol that has $j$ blocks of size $2$ and all other blocks of size $1$.  We can take this cell to be a $j$--dimensional cube.

\begin{theorem} \label{thm:curvature}
The cubical complex $\cell(n,2)$ admits a locally-CAT(0) metric. As a corollary, $\config(n,2)$ is aspherical, i.e.,\ has a contractible universal cover. So $\pi_j (\config(n,2)) = 0$ for $j \ge 2$.
\end{theorem}

\begin{proof}
This follows immediately from Gromov's criterion for a cube complex to admit locally-CAT(0) metric \cite{Gromov87}. The only thing to check is that the link of every vertex in $\cell(n,2)$ is a ``flag'' simplicial complex. A simplicial complex is said to be flag if it is the clique complex of its underlying graph---i.e.,\ if it is maximal with respect to its $1$-skeleton. A precise statement and complete proof of Gromov's criterion can be found in Appendix I.6 of Davis's book \cite{Davis15}.

Checking that the link of a vertex in $\cell(n,2)$ is flag is straightforward. Let $v$ be a vertex in $\cell(n,2)$, and consider the link of $v$, $L=\mbox{lk}(v)$.  The vertex $v$ corresponds to a symbol in $\poset(n,2)$ where every block has width $1$.

The vertices of $L$ correspond to elements $\sigma \ge v$ in $\poset(n,2)$ where every block in $\sigma$ has width $1$ except one block of width $2$. Similarly, edges in $L$ correspond to symbols $\tau \ge v$ where every block in $\tau$ has width $1$ except two blocks of width $2$. These two blocks of width $2$ are disjoint pairs in $[n]$. %See Figure \ref{fig:link} for an example.

Suppose $v_1, v_2, \dots, v_k$ span a $k$-clique in $\mbox{lk}(v)$. Then every pair of vertices corresponds to a disjoint pair of elements in $[n]$, and then concatenating these $k$ disjoint pairs (and respecting the order within each pair) gives a symbol $C \in \symbols(n,2)$ with $k$ blocks of width $2$ and all the remaining blocks of width $1$. This symbol $C$ indexes a $k$-dimensional cube in $\cell(n,2)$, which corresponds to a $(k-1)$-dimensional simplex in $\mbox{lk}(v)$.

For example, let $v = (4 \mid 6 \mid 5 \mid 1 \mid 3 \mid 2)$. The vertices $(4 6 \mid 5 \mid 1 \mid 3 \mid 2)$, $(4 \mid 6 \mid 1 5 \mid 3 \mid 2)$, and $(4 \mid 6 \mid 5 \mid 1 \mid 2 3)$ span a clique in $\mbox{link}(v)$. Since the symbol $(4 6 \mid 1 5 \mid 2 3 )$ corresponds to a $3$-dimensional cube in $\cell(6,2)$, the clique in $\mbox{link}(v)$ is filled in by a $2$-dimensional face.  
\end{proof}

On the other hand, $\config(n,w)$ is not an Eilenberg--MacLane space if $3 \le w \le n-1$. Indeed, the $2$-skeleton of $\cell(n,w)$ is the same as the $2$-skeleton of $\cell(n)$ when $w$ is in this intermediate range. So if $\config(n,w)$ were a $K(\pi,1)$, its homology would have to agree with the configuration space of points but we will see in Section \ref{sec:regimes} that it does not.

\section{Asymptotic lower bounds} \label{sec:lower} 

In this section, we exhibit a large number of linearly independent cycles to prove lower bounds on Betti numbers. The following is well known.

\begin{lemma} \label{lem:lower} Suppose that $M$ is an open $d$-dimensional manifold, with submanifolds $Z_1, Z_2, \dots, Z_k$ and $Z_1^*, Z_2^*, \dots, Z_k^*$ satisfying the following.
\begin{enumerate}
\item Every $Z_i$ is a compact orientable $j$-dimensional submanifold without boundary,
\item every $Z^*_i$ is a closed orientable $(d-j)$-dimensional submanifold without boundary,
\item whenever $a \neq b$ we have that $Z_a \cap Z^*_b = \emptyset$, and
\item $Z_a$ intersects $Z_a^*$ transversely in a point for every $a$.
\end{enumerate}
Then for any choice of coefficient field for the homology, we have $\dim H_j ( M) \ge k$.
\end{lemma} 

\begin{proof}
Choose orientations of each $Z_i$ and let $[Z_i]$ in $H_j(M)$ be the fundamental class of $Z_i$.  Choose orientations of each $Z^*_i$ and let $[Z^*_i]$ in $H^{BM}_{d-i}(M)$ be the fundamental class of $Z^*_i$.  (Here $H^{BM}_{*}$ denotes homology with closed supports, or Borel--Moore homology.)

Choose an orientation of $M$ so that the intersection pairing
$$p: H_i(M) \times H^{BM}_{d-i}(M) \to \R$$ is defined. By the stated properties of the manifolds $Z_i$ and $Z^*_i$, this intersection pairing satisfies
\begin{itemize}
\item $p([Z_a] ,  [Z^*_b]) = 0$  for $a \neq b$,
\item $p([Z_a] ,  [Z^*_a]) = \pm 1.$
\end{itemize}
Therefore, the homology classes $[Z_1], [Z_2], \dots, [Z_k]$ are linearly independent in $H_i(M)$, so the dimension of $H_i(M)$ is at least $k$.\\
\end{proof}

%In Section \ref{sec:homotopy} it is shown that $U(n,w)$ is homotopy equivalent to $\config(n,w)$.  In what follows, we apply Lemma~\ref{lem:lower} to get a lower bound on the Betti numbers of $U(n,w)$, giving the same lower bound on the Betti numbers of $\config(n, w)$.
In what follows, rather than working with the space $\config(n, w)$ directly, it is most convenient to apply Lemma~\ref{lem:lower} to the space which we denote by $wU(n, w) - (0, \frac{w}{2})$, consisting of configurations of points in the strip $\mathbb{R} \times [-\frac{w}{2}, \frac{w}{2}]$ in which no $w+1$ points have the same $x$--coordinate.  Because $wU(n, w) - (0, \frac{w}{2})$ is an open subset of $\mathbb{R}^{2n}$, it is an open manifold and thus is appropriate for Lemma~\ref{lem:lower}.  It is homeomorphic to $U(n, w)$, which we have shown in Section~\ref{sec:homotopy} is homotopy equivalent to $\config(n, w)$.  Lemma~\ref{lem:lower} gives a lower bound on the Betti numbers of $wU(n, w) - (0, \frac{w}{2})$, giving the same lower bound on the Betti numbers of $\config(n, w)$.

\bigskip

\begin{definition}
Let $j = q(w-1) + r$ with $0 \le r < w-1$. A \emph{special symbol} $\alpha \in \symbols(n,w)$ is a symbol $(c_1 \mid c_2 \mid \dots \mid c_{m})$ such that

\begin{enumerate}
\item $\alpha$ has $q$ blocks of width $w$, $r$ blocks of width $2$, and all other blocks of width $1$,
\item in every block, the largest element appears first, and 
\item if $c_i$ and $c_{i+1}$ are consecutive blocks of width strictly less than $w$, then the first element of block $c_i$ is greater than the first element of block $c_{i+1}$.\\
\end{enumerate}
\end{definition}

Which symbols are special depends on $n$, $j$, $w$, but for the sake of simplicity we suppress these in the notation since these parameters are always implicit.

\begin{definition} \label{def:Z^a} For every special symbol $\alpha$, we define a closed submanifold $Z^*_{\alpha}$ in $wU(n,w) - (0, \frac{w}{2})$ as follows. 
\begin{enumerate}
\item If $\sigma_k$ and $\sigma_\ell$ are in the same block and $k < \ell$, then $x_{\sigma_k} = x_{\sigma_\ell}$ and $y_{\sigma_k} > y_{\sigma_\ell}$.
\item If $\sigma_k$ and $\sigma_\ell$ are in different blocks and $k < \ell$, and either $\sigma_k$ or $\sigma_\ell$ is in a block of width $w$, then $x_{\sigma_k} < x_{\sigma_\ell}$.
\end{enumerate}
\end{definition}

\begin{figure}
\begin{tikzpicture}[scale=0.5]
\node[blue] at (0.2,2.5)  {\bf 5};
\node[blue]  at (1,1)  {\bf 4};
\node[blue]  at (1,3) {\bf 13};
\node[blue]  at (2,2.8) {\bf 19};
\node[blue]  at (5,0) {\bf 7};
\node[blue]  at (5,1) {\bf 1};
\node[blue]  at (5,2) {\bf 11};
\node[blue]  at (5,3) {\bf 6};
\node[blue]  at (5,4) {\bf 23};
\node[blue]  at (16,2.5) {\bf 17};
\node[blue]  at (10.5,2.75) {\bf 14};
\node[blue]  at (11.5,1.5) {\bf 8};
\node[blue]  at (11.5,3) {\bf 9};
\node[blue]  at (13,2) {\bf 3};
\node[blue]  at (13,3) {\bf 10};
\node[blue]  at (18,-0.04) {\bf 2};
\node[blue]  at (18,0.87) {\bf 12};
\node[blue]  at (18,1.825) {\bf 15};
\node[blue]  at (18,2.78) {\bf 18};
\node[blue]  at (18,3.69) {\bf 21};
\node[blue]  at (21,1) {\bf 16};
\node[blue]  at (21,3) {\bf 22};
\node[blue]  at (19.75,0.4) {\bf 20};
\draw (-1,-0.5)--(23,-0.5);
\draw (-1,4.5)--(23,4.5);

\draw[blue] (0.2,2.05) -- (0.2,2.95);
\draw[blue] (1,0.55) -- (1,1.45);
\draw[blue] (1,2.55) -- (1,3.45);
\draw[blue] (2,2.35) -- (2,3.25);
\draw[blue] (5,-0.48) -- (5,0.42);
\draw[blue] (5,0.43) -- (5,1.33);
\draw[blue] (5,1.55) -- (5,2.45);
\draw[blue] (5,2.69) -- (5,3.58);
\draw[blue] (5,3.59) -- (5,4.49);
%\draw[blue, thick] (9,0) circle (0.5);
\draw[blue] (16,2.05) -- (16,2.95);
%\draw[blue, thick] (10,0) circle (0.5);
\draw[blue] (10.5,2.3)--(10.5,3.2);
\draw[blue] (11.5,1.05)--(11.5,1.95);
\draw[blue] (11.5,2.55)--(11.5,3.45);
\draw[blue] (13,1.55)--(13,2.45);
\draw[blue] (13,2.55)--(13,3.45);
\draw[blue] (18,-0.49)--(18,0.41);
\draw[blue] (18,0.42)--(18,1.32);
\draw[blue] (18,1.33)--(18,2.32);
\draw[blue] (18,2.33)--(18,3.23);
\draw[blue] (18,3.24)--(18,4.14);
\draw[blue] (21,0.55)--(21,1.45);
\draw[blue] (21,2.55)--(21,3.45);
\draw[blue] (19.75,-0.05)--(19.75,0.85);

\end{tikzpicture}

\vspace{5pt}

\begin{tikzpicture}[scale=0.5]
%\draw[blue] (0,1.55) -- (0,2.45);
%\draw[blue](1.913049147247882,1.831763024468786) -- (1.913049147247882,2.731763024468786) ;
%\draw[blue](1.086950852752118, 1.331763024468786) --(1.086950852752118, 2.168236975531214);
%\draw[blue] (3,1.55) -- (3,2.45);
%\draw[blue] (9,1.55) -- (9,2.45);
%\draw[blue] (10,1.55)--(10,1.95);
%\draw[blue] (11.01,1.579)--(11.01,2.479);
%\draw[blue] (11.999,1.521)--(11.999,2.421);
%\draw[blue] (13.584,1.143)--(13.584,1.957);
%\draw[blue] (13.416,1.943)--(13.416,2.843);
%%\draw[blue] (17,-0.45)--(17,0.45);
%%\draw[blue] (17,0.55)--(17,1.45);
%%\draw[blue] (17,1.55)--(17,2.45);
%%\draw[blue] (17,2.55)--(17,3.45);
%%\draw[blue] (17,3.55)--(17,4.45);
%\draw[blue] (20.887,1.836)--(20.887,2.736);
%\draw[blue] (20.113,1.366)--(20.113,2.134);
%\draw[blue] (22,1.55)--(22,2.45);

\draw[blue, thick] (0,2) circle (0.5);
\draw[blue, thick](1.913,2.282) circle (0.5);
\draw[blue, thick] (1.087,1.718) circle (0.5);
\draw(1.5,2) circle (1);
\draw[blue, thick] (3,2) circle (0.5);
%\draw[blue, thick] (6,0) circle (0.5);
%\draw[blue, thick] (6,1) circle (0.5);
%\draw[blue, thick] (6,2) circle (0.5);
%\draw[blue, thick] (6,3) circle (0.5);
\draw[blue, thick] (7.933,2.514) circle (0.5);
%\draw[blue] (7.933,2.064)--(7.933,2.964);
\draw(6,2) circle (2.5);
\coordinate (A) at (5.517,1.871);
\draw(A) circle (2);
\draw(A) ++ (235:1.5) coordinate (B);
\draw[blue,thick](B) circle (0.5);
%\draw[blue](B) ++ (0,-0.45) --++ (0,0.9);
\node[blue] at (B) {\bf 1};
\draw(A) ++ (235:-0.5) coordinate (B2);
\draw(B2) circle (1.5);
\draw(B2) ++ (23:1) coordinate (C);
\draw[blue,thick](C) circle (0.5);
%\draw[blue](C) ++ (0,-0.45) --++ (0,0.9);
\node[blue] at (C) {\bf 11};
\draw(B2) ++ (23:-0.5) coordinate (C2);
\draw(C2) circle (1);
\draw(C2) ++ (160.5:0.5) coordinate(D1);
\draw(C2) ++ (160.5:-0.5) coordinate(D2);
\draw[blue,thick](D1) circle (0.5);
\node[blue] at (D1) {\bf 6};
\node[blue] at (D2) {\bf 23};
\draw[blue,thick](D2) circle (0.5);
%\draw[blue](D1) ++ (0,-0.45) --++ (0,0.9);
%\draw[blue](D2) ++ (0,-0.45) --++ (0,0.9);
\draw[blue, thick] (9,2) circle (0.5);
\draw[blue, thick] (10,2) circle (0.5);
\draw[blue, thick] (11.01,2.029) circle (0.5);
\draw[blue, thick] (11.999,1.971) circle (0.5);
\draw(11.5,2) circle (1);
\draw[blue, thick] (13.416,2.493) circle (0.5);
\draw[blue, thick] (13.584,1.507) circle (0.5);
\draw(13.5,2) circle (1);
%\draw[blue, thick] (17,0) circle (0.5);
%\draw[blue, thick] (17,1) circle (0.5);
%\draw[blue, thick] (17,2) circle (0.5);
%\draw[blue, thick] (17,3) circle (0.5);
\draw(17,2) circle (2.5) coordinate(A);
\draw(A)++(255:2) coordinate(A2);
\draw[blue, thick] (A2) circle (0.5);
\draw(A) ++(255:-0.5) coordinate(B);
\draw (B) circle(2);
\draw (B) ++(108:1.5) coordinate(C1);
\draw (B) ++(108:-0.5) coordinate(C2);
\draw[blue,thick] (C1) circle (0.5);
\draw (C2) circle (1.5);
\draw (C2) ++(12.5:1) coordinate(D1);
\draw (C2) ++(12.5:-0.5) coordinate(D2);
\draw[blue,thick] (D1) circle (0.5);
\draw (D2) circle (1);
\draw (D2) ++(306:0.5) coordinate(E1);
\draw (D2) ++(306:-0.5) coordinate(E2);
\draw[blue,thick] (E1) circle (0.5);
\draw[blue,thick] (E2) circle (0.5);
\node[blue] at (A2) {\bf 2};
\node[blue] at (C1) {\bf 12};
\node[blue] at (D1) {\bf 15};
\node[blue] at (E1) {\bf 18};
\node[blue] at (E2) {\bf 21};
%\draw[blue] (A2) ++(0,-0.45) --++(0,0.9);
%\draw[blue] (C1) ++(0,-0.45) --++(0,0.9);
%\draw[blue] (D1) ++(0,-0.45) --++(0,0.9);
%\draw[blue] (E1) ++(0,-0.45) --++(0,0.9);
%\draw[blue] (E2) ++(0,-0.45) --++(0,0.9);
%\draw(17,1.5) circle (2);
%\draw(17,1) circle (1.5);
%\draw(17,0.5) circle (1);
\draw[blue, thick] (20.881,2.316) circle (0.5);
\draw[blue, thick] (20.113,1.684) circle (0.5);
\draw(20.5,2) circle (1);
\draw[blue, thick] (22,2) circle (0.5);

\node[blue] at (0,2)  {\bf 19};
\node[blue] at (1.086950852752118,1.718236975531214)   {\bf 4};
\node[blue] at (1.913049147247882,2.281763024468786) {\bf 13};
\node[blue] at (3,2) {\bf 5};
%\node[blue] at (6,0) {\bf 7};
%\node[blue] at (6,1) {\bf 1};
%\node[blue] at (6,2) {\bf 11};
%\node[blue] at (6,3) {\bf 6};
\node[blue] at (7.933,2.514) {\bf 7};
\node[blue] at (9,2) {\bf 17};
\node[blue] at (10,2) {\bf 14};
\node[blue]  at (11.01,2.029) {\bf 3};
\node[blue]  at (11.999,1.971)  {\bf 10};
\node[blue] at (13.416,2.493) {\bf 8};
\node[blue] at (13.584,1.507) {\bf 9};
%\node[blue] at (17,0) {\bf 2};
%\node[blue] at (17,1) {\bf 12};
%\node[blue] at (17,2) {\bf 15};
%\node[blue] at (17,3) {\bf 18};
%\node[blue] at (17,4) {\bf 21};
\node[blue] at (20.113,1.684) {\bf 16};
\node[blue] at (20.887,2.316) {\bf 22};
\node[blue] at (22,2) {\bf 20};
\draw (-1,-0.5)--(23,-0.5);
\draw (-1,4.5)--(23,4.5);

\end{tikzpicture}

\vspace{5pt}

\begin{tikzpicture}[scale=0.5]
\draw[blue] (0,1.55) -- (0,2.45);
\draw[blue](1.5,1.05)--(1.5,1.95);
\draw[blue](1.5,2.05)--(1.5,2.95);
\draw[blue] (3,1.55) -- (3,2.45);
\draw[blue] (6,-0.45) -- (6,0.45);
\draw[blue] (6,0.55) -- (6,1.45);
\draw[blue] (6,1.55) -- (6,2.45);
\draw[blue] (6,2.55) -- (6,3.45);
\draw[blue] (6,3.55) -- (6,4.45);
\draw[blue] (9,1.55) -- (9,2.45);
\draw[blue] (10,1.55)--(10,2.45);
\draw[blue] (11.5,1.05)--(11.5,1.95);
\draw[blue] (11.5,2.05)--(11.5,2.95);
\draw[blue] (13.5,1.05)--(13.5,1.95);
\draw[blue] (13.5,2.05)--(13.5,2.95);
\draw[blue] (17,-0.45)--(17,0.45);
\draw[blue] (17,0.55)--(17,1.45);
\draw[blue] (17,1.55)--(17,2.45);
\draw[blue] (17,2.55)--(17,3.45);
\draw[blue] (17,3.55)--(17,4.45);
\draw[blue] (20.5,1.05)--(20.5,1.95);
\draw[blue] (20.5,2.05)--(20.5,2.95);
\draw[blue] (22,1.55)--(22,2.45);

\draw[blue, thick] (0,2) circle (0.5);
\draw[blue, thick](1.5,1.5) circle (0.5);
\draw[blue, thick] (1.5,2.5) circle (0.5);
\draw(1.5,2) circle (1);
\draw[blue, thick] (3,2) circle (0.5);
\draw[blue, thick] (6,0) circle (0.5);
\draw[blue, thick] (6,1) circle (0.5);
\draw[blue, thick] (6,2) circle (0.5);
\draw[blue, thick] (6,3) circle (0.5);
\draw[blue, thick] (6,4) circle (0.5);
\draw(6,3.5) circle (1);
\draw(6,3) circle (1.5);
\draw(6,2.5) circle (2);
\draw(6,2) circle (2.5);
\draw[blue, thick] (9,2) circle (0.5);
\draw[blue, thick] (10,2) circle (0.5);
\draw[blue, thick] (11.5,1.5) circle (0.5);
\draw[blue, thick] (11.5,2.5) circle (0.5);
\draw(11.5,2) circle (1);
\draw[blue, thick] (13.5,1.5) circle (0.5);
\draw[blue, thick] (13.5,2.5) circle (0.5);
\draw(13.5,2) circle (1);
\draw[blue, thick] (17,0) circle (0.5);
\draw[blue, thick] (17,1) circle (0.5);
\draw[blue, thick] (17,2) circle (0.5);
\draw[blue, thick] (17,3) circle (0.5);
\draw[blue, thick] (17,4) circle (0.5);
\draw(17,3.5) circle (1);
\draw(17,3) circle (1.5);
\draw(17,2.5) circle (2);
\draw(17,2) circle (2.5);
\draw[blue, thick] (20.5,1.5) circle (0.5);
\draw[blue, thick] (20.5,2.5) circle (0.5);
\draw(20.5,2) circle (1);
\draw[blue, thick] (22,2) circle (0.5);

\node[blue] at (0,2)  {\bf 19};
\node[blue] at (1.5,1.5)   {\bf 4};
\node[blue] at (1.5,2.5) {\bf 13};
\node[blue] at (3,2) {\bf 5};
\node[blue] at (6,0) {\bf 7};
\node[blue] at (6,1) {\bf 1};
\node[blue] at (6,2) {\bf 11};
\node[blue] at (6,3) {\bf 6};
\node[blue] at (6,4) {\bf 23};
\node[blue] at (9,2) {\bf 17};
\node[blue] at (10,2) {\bf 14};
\node[blue] at (11.5,1.5) {\bf 3};
\node[blue] at (11.5,2.5) {\bf 10};
\node[blue] at (13.5,1.5) {\bf 8};
\node[blue] at (13.5,2.5) {\bf 9};
\node[blue] at (17,0) {\bf 2};
\node[blue] at (17,1) {\bf 12};
\node[blue] at (17,2) {\bf 15};
\node[blue] at (17,3) {\bf 18};
\node[blue] at (17,4) {\bf 21};
\node[blue] at (20.5,1.5) {\bf 16};
\node[blue] at (20.5,2.5) {\bf 22};
\node[blue] at (22,2) {\bf 20};
\draw (-1,-0.5)--(23,-0.5);
\draw (-1,4.5)--(23,4.5);

\end{tikzpicture}
\caption{Upper picture: a point on the closed $34$-dimensional submanifold $Z^*_{\alpha}$.  Middle picture: a point on the compact $12$-dimensional submanifold $Z_\alpha$. Bottom picture: the single point of transverse intersection $Z^*_{\alpha} \cap Z_\alpha$. In all three pictures, $\alpha \in \symbols(23,5)$ is the special symbol 
$$\hspace{-0.75in} \alpha = (19 \mid 13 \  4 \mid 5 \mid 23  \  6 \ 11 \ 1 \ 7 \mid 17 \mid 14 \mid 10 \ 3 \mid 9 \  8 \mid 21 \ 18 \ 15 \ 12 \ 2 \mid 22 \ 16 \mid 20).$$} \label{fig:cycle}
\end{figure}
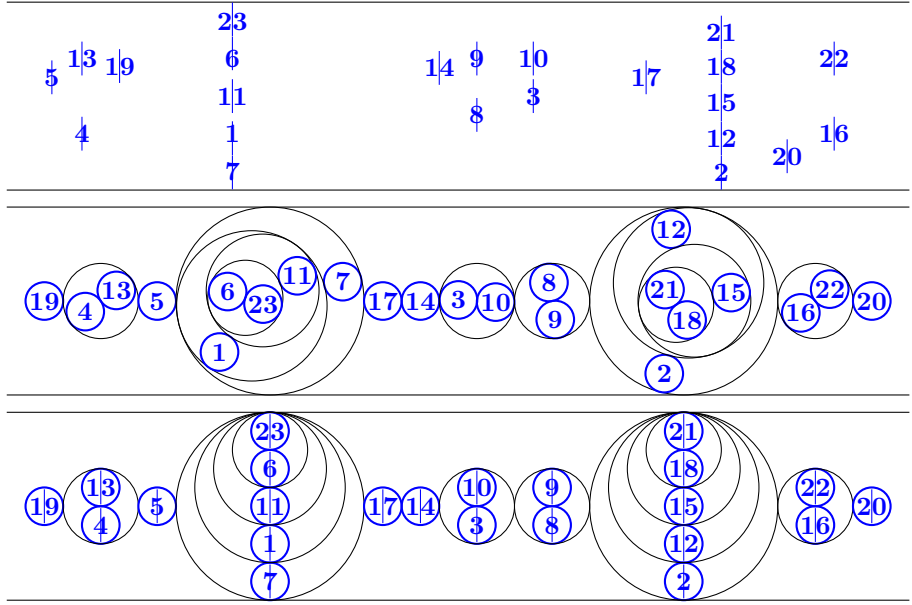

See, for example, the upper picture in Figure \ref{fig:cycle}. It is clear that every $Z^*_{\alpha}$ is closed in $wU(n,w) - (0, \frac{w}{2})$, by Definition \ref{def:unw}. Indeed, the strict inequalities $y_{\sigma_k} > y_{\sigma_\ell}$ and $x_{\sigma_k} < x_{\sigma_\ell}$ in Definition \ref{def:Z^a} could be replaced by weak inequalities $y_{\sigma_k} \ge y_{\sigma_\ell}$ and $x_{\sigma_k} \le x_{\sigma_\ell}$, since at most $w$ of the $x$-coordinates can be equal.\\

%We need to check is that every $Z^*_{\alpha}$ is closed in $U(n,w)$. All the inequalities that define the submanifold are clearly closed, except possibly the condition $x_{\sigma_k} < x_{\sigma_\ell}$. However this inequality is also actually closed. Indeed, since the assumption is that at least one of the two adjacent blocks has width $w$ and since $\epsilon < 1$, there is not room for another vertical interval to merge. So one could replace this inequality by $x_{\sigma_k} \le x_{\sigma_\ell}$.

Now, for every special symbol $\alpha$ we describe a cycle with the desired intersection properties. 

\begin{theorem}
Given a special symbol $\alpha$, there exists a cycle $Z_{\alpha}$ represented by an $j$-dimensional torus embedded in $wU(n,w) - (0, \frac{w}{2})$, such that whenever $\alpha' \neq \alpha$ we have that $Z_\alpha \cap Z^*_{\alpha'} = \emptyset$, and such that for every $\alpha$, $Z_{\alpha}$ intersects $Z^*_{\alpha}$ transversely in a point.
\end{theorem}

\begin{proof}
%By definition, $U(n,w)$ is an open subset of $\R^{2n}$, hence is an open manifold. % On the other hand, it makes more sense to think of the cycle $Z_\alpha$ as living in the disk configuration space $\config(n, w)$.  %We describe how to construct a cycle in $\config(n, w)$.  Then we simply scale up each configuration in the cycle by a factor of $(1 + \frac{\epsilon}{2w})$ in order to get our desired cycle $Z_\alpha$ in $U(n,w)$.

The cycle $Z_\alpha$ actually lies in the configuration space of disks of diameter $1$, which we can denote by $\config(n, w) - (0, \frac{w}{2})$ and which is a subspace of $wU(n, w) - (0, \frac{w}{2})$.  The general idea of the construction is illustrated in the middle picture in Figure~\ref{fig:cycle}. For every block $c_i$, we construct a $(w(c_i)-1)$-dimensional torus in the configuration space of only the disks appearing in that block.  Then we put the configurations for the different blocks horizontally next to each other in the strip in sequential order, making a $j$-parameter family of configurations that is an embedded $j$-dimensional torus.  

Looking more closely, for each individual block $c_i$, the corresponding $(w(c_i) - 1)$-dimensional torus is constructed roughly as follows.  We can spin the first two disks around each other inside a disk of diameter $2$.  Then we can spin the third disk around the first two disks, all inside a disk of diameter $3$.  Iterating this process, the final result is a disk of diameter $w(c_i)$, with the final disk of the block circling around its inner edge, and with the remaining disks moving around inside a disk of diameter $w(c_i) - 1$ tangent to the final disk.  In this way, $w(c_i)$ disks can move around inside a disk of diameter $w(c_i)$ to make a $(w(c_i)-1)$-dimensional torus in their configuration space.

More precisely, we construct the cycle $Z_\alpha$ as follows.  We parameterize the torus as
$$(S^{1})^{j} = \{  (\theta_1, \theta_2, \dots, \theta_{j}) \mid \theta_i \in [0,2 \pi), i =1, 2, \dots, j \}.$$
Given a symbol $\alpha$ and angles  $(\theta_1, \theta_2, \dots, \theta_{j})$, we need to specify a configuration in $\config(n, w) - (0, \frac{w}{2})$.

%For convenience of notation, we will describe a cycle in a strip of width $w$ as being centered at the $x$-axis rather than just above it; that is, the cycle will be embedded in the configuration space of points in the strip is $\mathbb{R} \times [-\frac{w}{2}, \frac{w}{2}] \subseteq \mathbb{R}^2$.  The rescaling and translating will give a cycle in $U(n,w)$. %Our configuration is $(x_1, y_1, \dots, x_n, y_n) \in \R^{2n}$ such that the unit-diameter disks around $(x_1, y_1), \ldots, (x_n, y_n)$ fall within the strip. 

We compute the coordinates $x_1, y_1, \ldots, x_n, y_n$ as follows. Let $\alpha=(c_1 \mid c_2 \mid \dots \mid c_m)$ be a special symbol. Let $w(c_i)$ denote the width of $c_i$. By the definition of special symbol, the blocks in $\alpha$ are all of width $1$, $2$, or $w$.

Let 
\begin{align*}
X_1 & = w(c_1)/2\\
X_2 & = w(c_1) + w(c_2) / 2\\
\dots \\
X_i &= w(c_1) + w(c_2) + \dots + w(c_{i-1})+ w(c_i) / 2\\
\end{align*}

This variable tells us how far to horizontally shift the center of the torus for the next block.  So, the imagined disk of diameter $w(c_i)$ that contains the disks of the $i$th block has center at $(X_i, 0)$.

Let $$D_i = w(c_1) + w(c_2) + \dots +w(c_i) - i .$$
This is a counter which tells us which angle we are on.  That is, the first $w(c_1) - 1$ angles correspond to the first block, the next $w(c_2) - 1$ angles correspond to the second block, and so on, so that angle $D_i$ is the last angle corresponding to the $i$th block.

\begin{enumerate}
\item If $w(c_i) =1$, that is, $c_i$ is the block with a single permutation element $c_i=|\sigma_m|$, then set 
$$(x_{\sigma_m}, y_{\sigma_m})= (X_i, 0).$$
\item If $w(c_i)=2$, that is,\ $c_i = | \sigma_m \sigma_{m+1}|$, then set
$$(x_{\sigma_m}, y_{\sigma_m})= (X_i - \frac{1}{2} \cos \theta_{D_i}, -\frac{1}{2} \sin \theta_{D_i}),$$
$$(x_{\sigma_{m+1}}, y_{\sigma_{m+1}}) = (X_i + \frac{1}{2} \cos \theta_{D_i},  \frac{1}{2} \sin \theta_{D_i}),$$
so that $\theta_{D_i}$ is the direction of the vector from disk $\sigma_{m}$ to disk $\sigma_{m+1}$.
\item If $w(c_i)=w$, i.e.,\ $c_i = | \sigma_m \sigma_{m+1} \dots \sigma_{m+w-1}|$, then
% got to here
\begin{enumerate}
\item Initialize $(u_0, v_0) = (X_i,0)$.
\item For $k=1,2, \dots, w$, we let $(u_k, v_k)$ be the center of the imagined disk containing the first $w-k$ disks of the block, and $\theta_{D_i + 1 - k}$ will be the direction of the vector from $(u_k, v_k)$ to the $k$th-to-last disk of the block.  That is, for $k = 1$ we let
\[(u_1, v_1) = (u_0, v_0) - \frac{1}{2}(\cos \theta_{D_i}, \sin \theta_{D_i}),\]
and
\[(x_{\sigma_{m+w-1}}, y_{\sigma_{m+w-1}}) = (u_0, v_0) + \frac{w-1}{2}(\cos \theta_{D_i}, \sin \theta_{D_i}),\]
and for general $k$ we let
$$(u_{k},v_{k}) = (u_{k-1},v_{k-1}) - \frac{1}{2} (\cos \theta_{D_i + 1 - k}, \sin \theta_{D_i + 1 - k}),$$
and
$$(x_{\sigma_{m + w - k}},y_{\sigma_{m + w - k}}) = (u_{k-1},v_{k-1}) + \frac{w-k}{2} (\cos \theta_{D_i+1-k} , \sin \theta_{D_i+1-k}).$$
\end{enumerate}\end{enumerate}

%To finish the construction of the cycle $Z_{\alpha}$, a $j$-dimensional torus in the configuration space of intervals $U(n, w)$, we dilate each configuration in the above construction by a factor of $w$ and translate vertically by $\frac{w}{2}$. There are never more than $w$ points with the same $x$-coordinate, so the resulting torus is in $U(n,w)$, as desired.
This completes the construction of the cycle $Z_\alpha$.
\\

Now we must check that whenever $\alpha' \neq \alpha$ we have that $Z_\alpha \cap Z^*_{\alpha'} = \emptyset$, and that $Z_{\alpha}$ intersects $Z^*_{\alpha}$ transversely in a point for every $\alpha$.

Suppose that $$p =(x_1, y_1, \dots, x_n, y_n) \in Z_\alpha \cap Z^*_{\alpha'}.$$ Define an equivalence relation on $[n]$ by setting $k \sim \ell $ if $x_k = x_\ell$ in $p$. By the definition of cycle $Z_{\alpha}$, if $k \sim \ell$ then $k$ and $\ell$ are in the same block of $\alpha$. By the definition of $Z^*_{\alpha'}$, if $k$ and $\ell$ are in the same block of $\alpha'$, then $k \sim \ell$. So then if $p \in Z_\alpha \cap Z^*_{\alpha'},$
if $k$ and $\ell$ are in the same block of $\alpha'$, then they are in the same block of $\alpha$. 

By assumption, both $\alpha$ and $\alpha'$ are special symbols in $\symbols(n,w)$, so they both have $q$ blocks of width $j$, $r$ blocks of width $2$ and the remaining blocks of width one. So it must be that the converse is also true, that if  $k$ and $\ell$ are in the same block of $\alpha$, then they are in the same block of $\alpha'$. 

Moreover, the partition of $[n]$ given by the equivalence relation $\sim$ must be the same as the partition into blocks given by $\alpha$ and $\alpha
'$. So the elements within every block are vertically aligned. In the special symbol $\alpha'$, the first element of a block is greatest in the underlying permutation, and in $Z^*_{\alpha'}$ it corresponds to the element at the top of the column (i.e.,\ has the largest $y$-coordinate).

Consider any point on $Z_\alpha$ in which the elements in each block are vertically aligned, with the greatest element of each block (and hence first element, because $\alpha$ is special) on top.  We claim that there is only one such point, and that the order of the elements in each block is the same as the order of their corresponding disks from top to bottom. 

The first element in the block is on top by assumption. Then since the first two elements are vertically aligned and in an imagined disk of diameter 2, the next element of the block must lie immediately below the first element. Continuing by induction, if the first $k$ elements of the block are vertically aligned and in an imagined disk of diameter  $k$, then the $k$th element of the block must be immediately below the $(k-1)$st element.  Thus, the configuration is completely determined by the assumption that the first element of the block is on top, proving that if $Z_\alpha$ and $Z^*_{\alpha'}$ intersect, then $\alpha = \alpha'$ and the intersection is a single point.

The only thing left to verify is that in this case the intersection is transverse. Since $Z_a$ and  $Z^*_{\alpha}$ intersect at a single point and are of complementary dimension in the ambient manifold, the claim of transversality is equivalent to checking that the tangent space to $U(n,w)$ is the direct sum of the tangent spaces to $Z_{\alpha}$ and $Z_{\alpha}^*$. Roughly, the tangent space to $Z^*_\alpha$ corresponds to the set of ways to assign a vector to each disk such that for each block, the horizontal components are equal; in the tangent space to $Z_\alpha$, the vector corresponding to each disk is horizontal and for each block there is a single linear relation on the vectors.  We omit the details.
\end{proof}

Finally, we are ready to prove lower bounds on the Betti numbers of $\config(n, w)$.

\begin{proof}[Proof of lower bounds for part (2) of Theorem \ref{thm:main}]

We have just verified the conditions of Lemma \ref{lem:lower}, which then implies that if $n \ge qw + 2r$ (i.e.,\ for sufficiently large $n$) then 
$$ \beta_j \ge  \binom{n}{\underbrace{w, \dots, w}_{q \mbox{ times}}, \underbrace{2, \dots, 2}_{r \mbox{ times}}, n-qw-2r}q!  \left( (w-1)! \right )^q \left(q+1\right)^{n-qw-r}.$$

This counts the number of special symbols in $\symbols(n,w)$. The multinomial coefficient counts the number of ways to partition $n$ into $q$ subsets of size $w$, $r$ subsets of size $2$, and $n-qw-2r$ subsets of size $1$. There are $q!$ ways to order the subsets of size $w$, and $\left( (w-1)! \right)^q$ ways to order the terms in each subset, considering the restriction that the largest element must come first within each part. Finally, we place the blocks of width $2$ and $1$ between the blocks of width $w$, and there are $(q+1)^{n-qw-r}$ ways to do this.

If  $j$ and $w$ are fixed and $n \to \infty$, then we write the simpler asymptotic expression
$$\beta_j \left[ \config(n,w) \right] = \Omega \left((q+1)^n n^{qw+2r} \right).$$
\end{proof}

Here $f = \Omega(g)$ means that there exists a positive constant $c$ such that
$f(n) \ge c g(n)$
for all sufficiently large $n$. \\

\section{The phase portrait for every $n$} \label{sec:solid-liquid} \label{sec:regimes}

In this section we prove Theorem \ref{thm:regimes}. Everything follows quickly from the homotopy equivalence $\config(n,w) \hequiv \cell(n,w)$ in Section \ref{sec:homotopy}, and the non-triviality of the cycles constructed in Section \ref{sec:lower}.\\

\begin{proof}[Proof of Theorem \ref{thm:regimes}]

\noindent

\begin{enumerate}

\item Gas: This is identical to the proof of (1) of Theorem \ref{thm:main}, in Section \ref{sec:consequences}. The proof of isomorphism on homology holds for every $n$.\\

\item Liquid: If $1 \le w \le n-1$ and $0 \le j \le n - \left\lceil n/w \right\rceil$ we see first that $H_j [ \config( n, w) ] \neq 0$. Indeed, the cycles constructed in Section \ref{sec:lower} are already enough. One can partition $[n]$ into at most $\lceil n/w \rceil$ blocks of width at most $w$. By ordering elements within a block, and reordering blocks if necessary, then we have a special symbol $\alpha$ with at most $\left\lceil n/w \right\rceil$ blocks. This indexes a cycle $Z_{\alpha}$ of dimension at least $n - \left\lceil n/w \right\rceil$.

We next see that if $j \ge w-1$ then the inclusion map $i: \config(n,w) \to \config(n,\R^2)$  does not induce an isomorphism on homology
$$i_*:\ H_j [ \config(n,w) ] \rightarrow H_j [ \config(n,\R^2)].$$
We observe that the kernel of $i_*$ is nontrivial. Consider two different torus cycles $Z_{\alpha}$ and $Z_{\alpha'}$, indexed by two different special symbols $\alpha,\alpha' \in \symbols(n,w)$ where $\alpha'$ is obtained from $\alpha$ by transposing two blocks (keeping the order of the elements within a block). Since $n \ge w + 1$ and $j \ge w-1$, this is always possible. Indeed, the condition that $j \ge w-1$ ensures that $\alpha$ and $\alpha'$ have at least one block of width $w$, and the condition that $n \ge w+1$ ensures that there is at least one other block. 

We have shown that $Z_{\alpha}$ and $Z_{\alpha'}$ are not homologous in $\config(n,w)$, but $i_* (Z_{\alpha})$ and $i_* (Z_{\alpha'})$ are homologous in $\config(n,\R^2)$, so we conclude that $Z_{\alpha} - Z_{\alpha'}$ is in the kernel of $i_*$.\\

\item Solid: Finally, we check that if $w \ge 1$ and 
$j \ge n -  \left\lceil \frac{n}{w} \right\rceil +1,$
then $$H_j \left[ \config( n , w) \right] = 0.$$
We know from Section \ref{sec:homotopy} that $\config(n,w) \sim \cell(n,w)$. The largest dimension of a cell in $\cell(n,w)$ is $n - \left\lceil \frac{n}{w} \right\rceil$, since the minimum number of blocks is $\left\lceil \frac{n}{w} \right\rceil$. So if $j \ge n -  \left\lceil \frac{n}{w} \right\rceil +1,$ then there are no $j$-dimensional cells, in which case there is no nontrivial $j$-dimensional homology. \\
\end{enumerate}
\end{proof} 

\section{Discrete Morse theory} \label{sec:Morse}

In this section, we describe a discrete gradient vector field on $\cell(n,w)$.  Then in the next section, we prove an upper bound on the number of critical cells, thus giving an upper bound on the Betti numbers.  This upper bound completes the proof of Theorem~\ref{thm:main}, the asymptotic rate of growth of Betti numbers.

A \emph{discrete vector field} $V$ on a regular CW complex $X$ is a collection of pairs of faces $[\alpha, \beta]$ where $\alpha$ is a face of $\beta$ and $\dim \alpha = \dim \beta - 1$, and such that every face can be in at most one pair.  The discrete vector field $V$ is said to be \emph{gradient} if there are no closed $V$--walks.  A $V$--walk is a collection of pairs of faces $[\alpha_1, \beta_1], [\alpha_2, \beta_2], \ldots, [\alpha_r, \beta_r]$ where $[\alpha_i, \beta_i] \in V$ for every $i$ and $\alpha_{i+1}$ is a codimension $1$ face of $\beta_i$ other than $\alpha_i$, and the $V$--walk is \emph{closed} if $\alpha_r = \alpha_1$.

We call a face \emph{critical} if it is not in any pair.  The fundamental theorem of discrete Morse theory~\cite{Forman02} is that $X$ is homotopy equivalent to a CW complex $X'$, where $X'$ has exactly one cell for every critical face in $V$.  Any discrete gradient vector field gives an upper bound on the Betti numbers of the cell complex: each Betti number is at most the number of critical cells in the corresponding dimension.  So, we give an asymptotic upper bound on the number of critical cells to get an asymptotic upper bound on the Betti numbers.  

We begin by describing which cells will be critical with respect to the discrete gradient vector field that we will construct.  In the symbol of a cell in $\cell(n, w)$, we say that a block is \emph{top-heavy} if the largest element of that block is the first element.  We designate some pairs of blocks as leader/follower pairs, as follows.  Roughly, we look at pairs of consecutive blocks, starting with the first two blocks and moving to the right.  We may designate a pair of blocks to be a leader (the first block) and a follower (the next block), in which case we look next at the two blocks after these, so that no follower gets immediately also labeled a leader, and the leader/follower pairs are disjoint.

More precisely, we say that a block is a \emph{leader} if it is not a follower and its first element is larger than all the other elements of that block and also all the elements of the next block; we say that a block is a \emph{follower} if the previous block is a leader.  These definitions allow us to describe the critical cells of our discrete gradient vector field.  We say that a cell of $\cell(n, w)$ is \emph{$k$--crit} if the following is true for the first $k$ blocks: every block that is not top-heavy is a follower, and every leader/follower pair has greater than $w$ elements, combined.  Our goal is to verify that this definition of $k$--crit agrees with which cells are critical with respect to the discrete gradient vector field we will construct.

\begin{theorem}\label{thm-v}
There is a discrete gradient vector field $V$ on $\cell(n, w)$, such that the critical cells are exactly those that are $k$--crit for all $k$.
\end{theorem}

In order to define the discrete vector field $V$, we describe how to find the matching cell for each non-critical cell of $\cell(n, w)$.  We define a function $v$ that sends each cell to its matching cell; that is, if $[\alpha, \beta]$ is a pair in $V$, then we will have $v(\alpha) = \beta$ and $v(\beta) = \alpha$, and for any critical cell $\alpha$, we will have $v(\alpha) = \alpha$.  The definition of $v$ is as follows.  Given a cell $\alpha$, if $\alpha$ is $k$--crit for all $k$, then we set $v(\alpha) = \alpha$.  Otherwise, we find $k$ such that $\alpha$ is $(k-1)$--crit but not $k$--crit.  There are two possibilities:
\begin{enumerate}
\item The $(k-1)$st block is a leader, the $k$th block is a follower, and their combined number of elements is at most $w$; or
\item The $k$th block is not a follower and is not top-heavy.
\end{enumerate}
We refer to the first case as the ``match-up at $k-1$'' case, and we refer to the second case as the ``match-down at $k$'' case.  In the first case, we obtain $v(\alpha)$ by swapping the $(k-1)$st block with the $k$th block and removing the bar between them.  In the second case, we obtain $v(\alpha)$ by adding a bar just before the largest element of the $k$th block, to separate it into two blocks, and then swapping those two blocks.  In order to be able to use $v$ to define $V$, we need to check that $v$ actually matches the cells in pairs.

\begin{lemma}\label{lemma-involution}
The function $v$ is an involution; that is, we have $v(v(\alpha)) = \alpha$ for every cell $\alpha$ of $\cell(n, w)$.
\end{lemma}

\begin{proof}
Suppose that $\alpha$ is a cell in the match-up at $k-1$ case.  We want to show that $v(\alpha)$ is in the match-down at $k-1$ case.  We know that $v(\alpha)$ is $(k-2)$--crit because $\alpha$ and $v(\alpha)$ agree in the first $k-2$ blocks.  Suppose for the sake of contradiction that block $k-1$ of $v(\alpha)$ is a follower.  Then block $k-1$ of $\alpha$ is also a follower, because in both cases the previous block is the same and the current block has the same largest element.  But we know that block $k-1$ of $\alpha$ is a leader and thus not a follower, giving a contradiction.  So block $k-1$ of $v(\alpha)$ is not a follower.  It is clear from the construction that block $k-1$ of $v(\alpha)$ is not top-heavy, so $v(\alpha)$ is in the match-down at $k-1$ case, and it is also clear from the construction that applying $v$ to $v(\alpha)$ gives $\alpha$ again.

Now suppose that $\alpha$ is a cell in the match-down at $k$ case.  We want to show that $v(\alpha)$ is in the match-up at $k$ case.  We know that $v(\alpha)$ is $(k-1)$--crit because $\alpha$ and $v(\alpha)$ agree in the first $k-1$ blocks.  To show that $v(\alpha)$ is $k$--crit, we need to check that block $k$ of $v(\alpha)$ is top-heavy and is not a follower.  It is clear from the construction that block $k$ of $v(\alpha)$ is top-heavy.  Suppose for the sake of contradiction that block $k$ of $v(\alpha)$ is a follower.  Then block $k$ of $\alpha$ is also a follower, because in both cases the previous block is the same and the current block has the same largest element.  But we know that block $k$ of $\alpha$ is not a follower, because $\alpha$ is in the match-down at $k$ case.  Thus block $k$ of $v(\alpha)$ cannot be a follower, and so $v(\alpha)$ is $k$--crit.  Knowing that block $k$ of $v(\alpha)$ is not a follower, it is clear from the construction that this block is a leader and that its combined number of elements with the next block is at most $w$, so $v(\alpha)$ is in the match-up at $k$ case.  Then it is also clear from the construction that applying $v$ to $v(\alpha)$ gives $\alpha$ again.

Thus if $\alpha$ is in any of the three cases---critical, match-up, or match-down---we have $v(v(\alpha)) = \alpha$.
\end{proof}

Having shown that every orbit of $v$ has either one or two elements, we can define $V$ to be the set of two-element orbits; that is, if $v(\alpha) = \beta$ and $v(\beta) = \alpha$, with $\beta \neq \alpha$, then the definition of $v$ implies that we may swap the labels if necessary so that $\alpha$ is a codimension $1$ face of $\beta$, and we let $[\alpha, \beta]$ be one of the pairs in $V$.  To finish the proof of Theorem~\ref{thm-v}, we need to show that $V$ is gradient.

\begin{lemma}\label{lemma-gradient}
The discrete vector field $V$ is gradient; that is, it does not admit any closed $V$--walks.
\end{lemma}

\begin{proof}
Suppose for the sake of contradiction that $[\alpha_1, \beta_1], [\alpha_2, \beta_2], \ldots, [\alpha_r, \beta_r]$ is a closed $V$--walk.  We define a function
\[\key \co \poset(n, w) \rightarrow \bigoplus_{i = 1}^\infty \mathbb{Z}\]
 and show that if we compare the various $\key(\alpha_i)$, they are in strictly decreasing lexicographical order.  This gives a contradiction with the assumption that the $V$--walk is closed with $\alpha_r = \alpha_1$.
 
 The key function is defined as follows.  Given the symbol $\alpha$ of a cell in $\cell(n, w)$, we consider each block, and we set entry $2k-1$ of $\key(\alpha)$ to be the first element of the $k$th block, unless that block is a follower, in which case we set that entry to be zero; in either case, we set entry $2k$ of $\key(\alpha)$ to be the number of elements of the $k$th block.  Past twice the number of blocks, all entries of $\key(\alpha)$ are zero.  The lexicographical order on $\bigoplus_{i = 1}^\infty \mathbb{Z}$ is defined as follows: to compare two elements, we find the first entry where they differ, and we order the elements by their values in $\mathbb{Z}$ at that entry.
 
 We claim that for any $i$, we have $\key(\alpha_{i+1}) < \key(\alpha_i)$.  Let $k$ be the block where $\alpha_i$ merges to make $\beta_i$; that is, $\alpha_i$ is match-up at $k$ and $\beta_i$ is match-down at $k$.  Some block $k'$ of $\beta_i$ is split to form $\alpha_{i+1}$, and there are three cases: it is the same block $k' = k$, it is an earlier block $k' < k$, or it is a later block $k' > k$.  
 
 Suppose $k' = k$.  We know that block $k$ of $\alpha_i$ is the longest subblock of block $k$ of $\beta_i$ that begins with the largest element of that block, so comparing entries $2k-1$ and $2k$ of $\key(\alpha_i)$ and $\key(\alpha_{i+1})$, we find $\key(\alpha_{i+1}) < \key(\alpha_i)$ in this case.  
 
 Suppose $k' < k$.  Because $\beta_i$ is $(k-1)$--crit, the block $k'$ that is split is either top-heavy or a follower, and block $k'$ of $\alpha_{i+1}$ is a subblock of block $k'$ of $\beta_i$.  In the top-heavy case, comparing at entries $2k'-1$ and $2k'$ gives $\key(\alpha_{i+1}) < \key(\beta_i)$, and because $\beta_i$ and $\alpha_i$ agree past block $k'$, this implies $\key(\alpha_{i+1})$.  In the follower case, block $k'$ of $\alpha_{i+1}$ remains a follower, so comparing at entry $2k' - 1$ gives $\key(\alpha_{i+1}) < \key(\beta_i)$ and so $\key(\alpha_{i+1}) < \key(\alpha_i)$.
 
 Suppose $k' > k$.  Then block $k$ of $\alpha_{i+1}$ is the same as block $k$ of $\beta_i$, which has a smaller first element than block $k$ of $\alpha_i$ (which is a leader and not a follower).  Thus, comparing at entry $2k-1$ gives $\key(\alpha_{i+1}) < \key(\alpha_i)$.
 
 Thus, in all three cases the sequence $\key(\alpha_i)$ is strictly decreasing and so cannot be a cycle, contradicting the existence of a closed $V$--walk, and so $V$ is gradient.
\end{proof}

Together, Lemma~\ref{lemma-involution} and Lemma~\ref{lemma-gradient} imply Theorem~\ref{thm-v}.

\begin{proof}[Proof of Theorem~\ref{thm-v}]
Lemma~\ref{lemma-involution} shows that the discrete vector field $V$ specified by the function $v$ is well-defined: each cell can be in at most one pair in $V$.  The construction of $v$ automatically implies that the critical cells of $V$ are those that are $k$--crit for all $k$, because those are the only cells that are fixed points of $v$.  Lemma~\ref{lemma-gradient} shows that the discrete vector field $V$ is gradient.
\end{proof}

\section{Asymptotic upper bounds} \label{sec:upper}

In order to finish the proof of Theorem \ref{thm:main}, we need to prove an asymptotic upper bound on the number of critical cells of each dimension.  To do this, we group the critical cells of each dimension $j$ into finitely many groupings and prove that each grouping satisfies the asymptotic bound.  The groupings are called skylines.  Roughly, the skyline retains the information about which blocks form leader/follower pairs and about the sequence of sizes of blocks, but forgets the numbers (corresponding to labels of disks) and all the blocks of size $1$ that are neither leaders nor followers.  Given the symbol of a critical cell in $\cell(n, w)$, we refer to each leader/follower pair as a $2$-block \emph{barrier}.  We find the \emph{skyline} of that cell by the following process: we delete all the blocks that have just one element and are neither leaders nor followers (along with a bar adjacent to each), we replace the first element of each leader block by $1$, and we replace all the other numbers in the symbol by $0$.  

The resulting skyline is a kind of symbol in which all of the numbers are $0$ or $1$.  If the original cell was $j$--dimensional, then $j$ is the number of zeros and ones in the skyline minus the number of blocks in the skyline, much as in the original cell.  Any block with only one element is part of a barrier, so there are only finitely many different skylines for each $j$, independent of $n$.  For each skyline $S$, we let $b(S)$ (``barriers'') denote the number of barriers, equal to the number of ones in $S$, and we let $z(S)$ (``zeros'') denote the number of zeros in $S$.  In preparation for proving Theorem \ref{thm:main}, the following lemma implies an upper bound on the number of critical cells with a given skyline.

\begin{lemma}\label{lemma-skyline}
For every skyline $S$, there is an injective function $\code_S$ from the set of critical cells with skyline $S$ into the set $[n]^{z(S)}\times [b(S) + 1]^n$.
\end{lemma}

\begin{proof}
The function $\code_S$ is defined as follows.  Given a critical cell $\alpha$ with skyline $S$, we can map $\alpha$ to an element of $[n]^{z(S)}$ by recording the original number in $\alpha$ corresponding to each zero in $S$, in the order these numbers appear in $\alpha$.  For the second coordinate, we divide the symbol of $\alpha$ into $b(S) + 1$ intervals: all the blocks up through the first barrier, all the blocks after the first barrier and up through the second barrier, and so on, with the last interval being all the blocks after the last barrier.  Then we can map $\alpha$ to an element of $[b(S) + 1]^n$ by recording, for each number in $\alpha$, which of the $b(S) + 1$ intervals it appears in.

To show that the function $\code_S$ is injective, we show how to recover $\alpha$ from $\code_S(\alpha)$.  The $[n]^{z(S)}$ coordinate specifies the original number for each $0$ in $S$, so what remains is to find the original number for each $1$ in $S$ and to figure out where to insert the remaining numbers as one-element blocks.  We can recover the original number for each $1$ in $S$ by finding which of the $b(S) + 1$ intervals ends with that barrier, selecting all the numbers in that interval, and taking the greatest of those numbers---the preceding blocks in the interval are top-heavy with initial elements in increasing order, and the $1$ corresponds to the initial element of a leader block.  Then, for all the numbers that do not correspond to zeros or ones in $S$, we find which of the $b(S) + 1$ intervals each number belongs to, and insert it as a one-element block into that interval in such a way that the initial elements of all the blocks in that interval (excluding the follower block at the end) are in increasing order.  Because we can use this process to recover $\alpha$ from $\code_S(\alpha)$, the function $\code_S$ is injective.
\end{proof}

Putting these bounds together for all finitely many skylines, we can finish the proof of Theorem \ref{thm:main}.

\begin{proof}[Proof of Theorem \ref{thm:main}]
The statements about the gas regime and the solid regime have already been addressed, and in Section~\ref{sec:lower} we have shown that if $j = q(w-1) + r$ with $q \geq 1$ and $0 \leq r < w-1$, then we have
\[\beta_j[\config(n, w)] = \Omega((q+1)^nn^{qw + 2r}).\]
Thus, what remains is to prove that in this case we also have
\[\beta_j[\config(n, w)] = O((q+1)^nn^{qw + 2r}).\]
Lemma~\ref{lemma-skyline} implies that for any skyline $S$, the number of critical cells with that skyline is at most $(b(S) + 1)^nn^{z(S)}$.  Because the Betti number $\beta_j$ is bounded by the number of critical cells of dimension $j$, and because there are finitely many skylines for each $j$, it then suffices to prove that for any skyline $S$ corresponding to $j$--dimensional cells, we have
\[(b(S) + 1)^nn^{z(S)} = O((q+1)^nn^{qw+2r}).\]
Thinking of each block of size $k$ as contributing a value of $k-1$ to $j$, we observe that each $2$--block barrier in $S$ contributes a combined value of at least $w-1$ to $j$.  Thus we have $b(S) \leq q$.  In the case where $b(S) < q$, we certainly have $(b(S) + 1)^nn^{z(S)} = O((q+1)^nn^{qw+2r})$, because the factor that is exponential in $n$ overwhelms the factor that is polynomial in $n$.

Thus, it suffices to prove that if $b(S) = q$, then $z(S) \leq qw + 2r$.  The number of zeros in $S$ is $j$ plus the number of blocks in $S$ without a $1$.  Because $j = q(w-1) + r$, this means that it suffices to show that the number of blocks in $S$ without a $1$ is at most $q + r$.  Each barrier contains exactly one block without a $1$, so there are $q$ such blocks.  The other blocks without a $1$ are not part of barriers, so they have size at least $2$.  Each of these contributes at least $1$ to $j$, and the barriers together contribute at least $q(w-1)$ to $j$, so there are at most $r$ of these non-barrier blocks in $S$.  Thus, together the number of blocks in $S$ without a $1$ is at most $q + r$, so we have $z(S) \leq qw + 2r$, and thus
\[\beta_j[\config(n, w)] \leq \#(\textrm{crit\ cells\ of\ dim\ }j) = O \left( (q+1)^n n^{qw + 2r} \right),\]
completing the proof of Theorem \ref{thm:main}.
\end{proof}

\section{Comments} \label{sec:comments}

\begin{enumerate}

\item In principle, one could compute homology of $\config(n,w)$ exactly. For example, $\config(n,w)$ is homotopy equivalent to $U(n,w)$, which in turn is homeomorphic to the complement of a certain real subspace arrangement. We define $\mathcal{A}_{n,w}$ to be the collection of ${n \choose 2}$ subspaces of codimension $2$
$$ \{ \left( x_1, y_1, \dots, x_n, y_n \right) \in \R^{2n} \mid x_k = x_{\ell} \mbox{ and } y_k = y_{\ell}  \}$$
together with the $n \choose w+1$ subspaces of codimension $w+1$ 
$$ \{ \left( x_1, y_1, \dots, x_n, y_n \right) \mid x_{i_1} = x_{i_2} = \dots = x_{i_{w+1}}  \}.$$

Since $\config(n, w)$ is homotopy equivalent to the complement of this subspace arrangement, the homology is determined by the intersection lattice of the arrangement \cite{GM88}. One might apply essentially combinatorial formulas to derive a formula for $\beta_j [\config(n,w)]$. Such an exact formula might be nice to have, even if in a complicated or recursive form. \\

\item The definitions of homological solid, liquid, and gas make sense even for $0$th homology, especially for bounded regions. Determining the threshold for the solid-liquid phase transition for $0$th homology passing from trivial to nontrivial is equivalent to the well-studied ``sphere packing'' problem. %The largest radius of spheres that will fit in the region corresponds to where the configuration space transitions from empty to nonempty. 

There is another transition for $0$th homology, the homological liquid-gas phase transition, where the configuration space becomes connected. This seems to be much less well studied, but the threshold for connectivity is a natural and important question for a number of reasons. For example, in his survey article~\cite{Diaconis09}, Diaconis writes about it in the context of ergodicity of Markov chains, a requirement for being able to effectively sample a configuration space by making small random movements of disks. See also \cite{Kahle12} for discussion of the connectivity threshold. \\

\item We show in Section \ref{sec:lower} that certain toruses generate a positive fraction of the homology, but on the other hand we also know that even if one considers all of the toruses that one can make in similar ways, they do not seem to generate all of the homology. Consider the example $n=3$, $w=2$, $j=1$, illustrated in Figure \ref{fig:giant}. We know that $\beta_1 = 7$, but only $6$ cycles are accounted for by rotating a pair of disks around each other, and having the third disk on either one side or the other. The ``outside circle'' in the figure is visibly not in the span of the six smaller cycles. \\

\item Discrete Morse theory has been studied on the Salvetti complex before. For a more geometric approach to discrete gradients on $\cell(n)$, see \cite{SS07}, \cite{MS11}, and \cite{LP18}. We do not know whether the techniques from these papers can improve the upper bounds on $\beta_j [ \config(n,w) ]$, or perhaps even produce perfect discrete Morse functions or minimal CW complexes for $\config(n,w)$. \\

\item A related family of spaces is the ``no $k$-equal space'' $M^{\R}_{n, k}$ studied by Bj\"orner et al.\ in \cite{BL94,BLY92,BW95}. In particular, there is a natural map $\config(n,w) \to M^{\R}_{n,w+1}$ by projecting onto the $x$-coordinates. We do not know much about the induced map on homology in general. We point here out a coincidence we notice in the data that we do not currently have a good explanation for.
 
Comparing Table 1 in our appendix with the first table in the appendix of \cite{BW95}, it seems possible that $H_1 ( \config(n,2))$ is isomorphic to $H_1( M^{\R}_{n+1,3} )$---at least the Betti numbers are equal for $n \le 8$.

We emphasize that $\config(n,2)$ is a configuration space of $n$ points, and $M^{\R}_{n+1,3}$ is a configuration space of $n+1$ points, so we do not even have an obvious candidate of map to induce such an isomorphism. Supposing that there were such a map, we might wonder if it also induces an isomorphism on $\pi_1$ but apparently not, as follows.

We showed that $\config(n,2)$ is a $K(\pi,1)$ in Section \ref{sec:consequences}. The question of whether $M^{\R}_{n,3}$ is a $K(\pi,1)$ was asked by Bj\"orner \cite{Bjorner94} and answered affirmatively by Khovanov \cite{Khovonov96}. Khovanov describes this as a real analogue of the fact that $M^{\mathbb{C}}_{n,2}$ (the configuration space of points in the plane) is a $K(\pi,1)$. Since both spaces are $K(\pi,1)$'s, if they had isomorphic fundamental groups then they would be homotopy equivalent. But the Betti number tables rule out the higher homology groups $j \ge 2$ being isomorphic.\\

\end{enumerate}

\section*{Appendix by Ulrich Bauer and Kyle Parsons \\}

\bigskip

We computed the Betti numbers $\beta_j \left[ \cell(n,w) \right]$ for $n \le 8$, for homology with $\Z /2$ coefficients,  using the software PHAT \cite{BKRW17}. The results of the computations appear in Table \ref{table:Betti}. For a point of reference, we note that $\cell(8)$ has over $5$ million cells.
  \\

\begin{table}[]
\caption{Betti numbers of $\config(n,w)$ for small $n$ and $w$. Bold font indicates that homology is in the ``liquid regime''.}
\label{table:Betti}
\begin{tabular}{|c|c||c|c|c|c|c|c|c|c|}
\hline
$n$ & $w$ & $\beta_0$ & $\beta_1$ & $\beta_2$ & $\beta_3$ & $\beta_4$ & $\beta_5$ & $\beta_6$ & $\beta_7$ \\ \hline
2 & 1 & \bf 2 & 0 & 0 & 0 & 0 & 0 & 0 & 0 \\
2 & 2 & 1 & 1 & 0 & 0 & 0 & 0 & 0 & 0 \\ \hline
3 & 1 & \bf 6 & 0 & 0 & 0 & 0 & 0 & 0 & 0 \\
3 & 2 & 1 & \bf 7 & 0 & 0 & 0 & 0 & 0 & 0 \\ 
3 & 3 & 1 & 3 & 2 & 0 & 0 & 0 & 0 & 0 \\ \hline
4 & 1 & \bf 24 & 0 & 0 & 0 & 0 & 0 & 0 & 0 \\
4 & 2 & 1 & \bf 31 & 6 & 0 & 0 & 0 & 0 & 0 \\
4 & 3 & 1 & 6 & \bf 29 & 0 & 0 & 0 & 0 & 0 \\ 
4 & 4 & 1 & 6 & 11 & 6 & 0 & 0 & 0 & 0 \\ \hline
5 & 1 & \bf 120 & 0 & 0 & 0 & 0 & 0 & 0 & 0 \\
5 & 2 & 1 & \bf 111 & \bf 110 & 0 & 0 & 0 & 0 & 0 \\
5 & 3 & 1 & 10 & \bf 169 & \bf 40 & 0 & 0 & 0 & 0 \\
5 & 4 & 1 & 10 & 35 & \bf 146 & 0 & 0 & 0 & 0 \\ 
5 & 5 & 1 & 10 & 35 & 50 & 24 & 0 & 0 & 0 \\ \hline
6 & 1 & \bf 720 & 0 & 0 & 0 & 0 & 0 & 0 & 0 \\
6 & 2 & 1 & \bf 351 & \bf 1160 & \bf 90 & 0 & 0 & 0 & 0 \\
6 & 3 & 1 & 15 & \bf 714 & \bf 780 & \bf 80 & 0 & 0 & 0 \\
6 & 4 & 1 & 15 & 85 & \bf 1066 & \bf 275 & 0 & 0 & 0 \\
6 & 5 & 1 & 15 & 85 & 225 & \bf 875 & 0 & 0 & 0 \\ 
6 & 6 & 1 & 15 & 85 & 225 & 274 & 120 & 0 & 0 \\ \hline
7 & 1 & \bf 5040 & 0 & 0 & 0 & 0 & 0 & 0 & 0 \\
7 & 2 & 1 & \bf 1023 & \bf 9212 & \bf 3150 & 0 & 0 & 0 & 0 \\
7 & 3 & 1 & 21 & \bf 2568 & \bf 6468 & \bf 3920 & 0 & 0 & 0 \\
7 & 4 & 1 & 21 & 175 & \bf 5272 & \bf 5957 & \bf 840 & 0 & 0 \\
7 & 5 & 1 & 21 & 175 & 735 & \bf 7678 & \bf 2058 & 0 & 0 \\
7 & 6 & 1 & 21 & 175 & 735 & 1624 & \bf 6084 & 0 & 0 \\ 
7 & 7 & 1 & 21 & 175 & 735 & 1624 & 1764 & 720 & 0 \\ \hline
8 & 1 & \bf 40320 & 0 & 0 & 0 & 0 & 0 & 0 & 0 \\
8 & 2 & 1 & \bf 2815 & \bf 61194 & \bf 60900 & \bf 2520 & 0 & 0 & 0 \\
8 & 3 & 1 & 28 & \bf 8385 & \bf 37464 & \bf 76146 & \bf 6720 & 0 & 0 \\
8 & 4 & 1 & 28 & 322 & \bf 21477 & \bf 54910 & \bf 36239 & \bf 2520 & 0 \\
8 & 5 & 1 & 28 & 322 & 1960 & \bf 43728 & \bf 49959 & \bf 7896 & 0 \\
8 & 6 & 1 & 28 & 322 & 1960 & 6769 & \bf 62525 & \bf 17101 & 0 \\
8 & 7 & 1 & 28 & 322 & 1960 & 6769 & 13132 & \bf 48348 & 0 \\ 
8 & 8 & 1 & 28 & 322 & 1960 & 6769 & 13132 & 13068 & 5040 \\ \hline
\end{tabular}
\end{table}

\bibliographystyle{plain}
\bibliography{disksrefs}

\begin{thebibliography}{10}

\bibitem{Alpert17}
Hannah Alpert.
\newblock Restricting cohomology classes to disk and segment configuration
  spaces.
\newblock {\em Topology Appl.}, 230:51--76, 2017.

\bibitem{Arnold14}
Vladimir~I. Arnold.
\newblock {\em The cohomology ring of the colored braid group}, pages 183--186.
\newblock Springer Berlin Heidelberg, Berlin, Heidelberg, 2014.

\bibitem{BBK14}
Yuliy Baryshnikov, Peter Bubenik, and Matthew Kahle.
\newblock Min-type {M}orse theory for configuration spaces of hard spheres.
\newblock {\em Int. Math. Res. Not. IMRN}, 9:2577--2592, 2014.

\bibitem{BEJM17}
Ulrich Bauer, Herbert Edelsbrunner, Grzegorz Jablonski, and Marian Mrozek.
\newblock \v{C}ech--{D}elaunay gradient flow and homology inference for
  self-maps.
\newblock arXiv:1709.04068, 2017.

\bibitem{BKRW17}
Ulrich Bauer, Michael Kerber, Jan Reininghaus, and Hubert Wagner.
\newblock {PHAT -- Persistent Homology Algorithms Toolbox}.
\newblock {\em Journal of Symbolic Computation}, 78:76--90, January 2017.
\newblock Software available at \url{https://bitbucket.org/phat-code/phat}.

\bibitem{Bjorner94}
Anders Bj\"{o}rner.
\newblock Subspace arrangements.
\newblock In {\em First {E}uropean {C}ongress of {M}athematics, {V}ol. {I}
  ({P}aris, 1992)}, volume 119 of {\em Progr. Math.}, pages 321--370.
  Birkh\"{a}user, Basel, 1994.

\bibitem{BL94}
Anders Bj\"{o}rner and L\'{a}szl\'{o} Lov\'{a}sz.
\newblock Linear decision trees, subspace arrangements and {M}\"{o}bius
  functions.
\newblock {\em J. Amer. Math. Soc.}, 7(3):677--706, 1994.

\bibitem{BLY92}
Anders Bj{\"o}rner, L{\'a}szl{\'o} Lov{\'a}sz, and A~Yao.
\newblock Linear decision trees: volume estimates and topological bounds.
\newblock In {\em Proc. 24th ACM Symp. on Theory of Computing}, volume 170,
  page 177. Citeseer, 1992.

\bibitem{BW95}
Anders Bj\"{o}rner and Volkmar Welker.
\newblock The homology of ``{$k$}-equal'' manifolds and related partition
  lattices.
\newblock {\em Adv. Math.}, 110(2):277--313, 1995.

\bibitem{BZ14}
Pavle V.~M. Blagojevi{\'c} and G{\"u}nter~M. Ziegler.
\newblock Convex equipartitions via equivariant obstruction theory.
\newblock {\em Israel J. Math.}, 200(1):49--77, 2014.

\bibitem{CGKM12}
Gunnar Carlsson, Jackson Gorham, Matthew Kahle, and Jeremy Mason.
\newblock Computational topology for configuration spaces of hard disks.
\newblock {\em Phys. Rev. E}, 85:011303, Jan 2012.

\bibitem{Davis15}
Michael~W. Davis.
\newblock The geometry and topology of {C}oxeter groups.
\newblock In {\em Introduction to modern mathematics}, volume~33 of {\em Adv.
  Lect. Math. (ALM)}, pages 129--142. Int. Press, Somerville, MA, 2015.

\bibitem{Deeley11}
Kenneth Deeley.
\newblock Configuration spaces of thick particles on a metric graph.
\newblock {\em Algebr. Geom. Topol.}, 11(4):1861--1892, 2011.

\bibitem{Diaconis09}
Persi Diaconis.
\newblock The {M}arkov chain {M}onte {C}arlo revolution.
\newblock {\em Bull. Amer. Math. Soc. (N.S.)}, 46(2):179--205, 2009.

\bibitem{FMN14}
Steve Ferry, Konstantin Mischaikow, and Vidit Nanda.
\newblock Reconstructing functions from random samples.
\newblock {\em J. Comput. Dyn.}, 1(2):233--248, 2014.

\bibitem{Forman02}
Robin Forman.
\newblock A user's guide to discrete {M}orse theory.
\newblock {\em S\'em. Lothar. Combin.}, 48:Art.\ B48c, 35, 2002.

\bibitem{GM88}
Mark Goresky and Robert MacPherson.
\newblock {\em Stratified {M}orse theory}, volume~14 of {\em Ergebnisse der
  Mathematik und ihrer Grenzgebiete (3) [Results in Mathematics and Related
  Areas (3)]}.
\newblock Springer-Verlag, Berlin, 1988.

\bibitem{Gromov87}
M.~Gromov.
\newblock Hyperbolic groups.
\newblock In {\em Essays in group theory}, volume~8 of {\em Math. Sci. Res.
  Inst. Publ.}, pages 75--263. Springer, New York, 1987.

\bibitem{Kahle12}
Matthew Kahle.
\newblock Sparse locally-jammed disk packings.
\newblock {\em Ann. Comb.}, 16(4):773--780, 2012.

\bibitem{Khovonov96}
Mikhail Khovanov.
\newblock Real {$K(\pi,1)$} arrangements from finite root systems.
\newblock {\em Math. Res. Lett.}, 3(2):261--274, 1996.

\bibitem{KKLS16}
Rob Kusner, W\"{o}den Kusner, Jeffrey~C. Lagarias, and Senya Shlosman.
\newblock Configuration spaces of equal spheres touching a given sphere: The
  twelve spheres problem.
\newblock In {\em New Trends in Intuitive Geometry, Bolyai Society Mathematical
  Studies}, pages 219--277. Springer-Verlag GMBH, Germany, 2018.

\bibitem{LP18}
Davide Lofano and Giovanni Paolini.
\newblock Euclidean matchings and minimality of hyperplane arrangements.
\newblock arXiv:1809.02476, 2018.

\bibitem{MS11}
Francesca Mori and Mario Salvetti.
\newblock ({D}iscrete) {M}orse theory on configuration spaces.
\newblock {\em Math. Res. Lett.}, 18(1):39--57, 2011.

\bibitem{Salvetti87}
M.~Salvetti.
\newblock Topology of the complement of real hyperplanes in {${\bf C}^N$}.
\newblock {\em Invent. Math.}, 88(3):603--618, 1987.

\bibitem{SS07}
Mario Salvetti and Simona Settepanella.
\newblock Combinatorial {M}orse theory and minimality of hyperplane
  arrangements.
\newblock {\em Geom. Topol.}, 11:1733--1766, 2007.

\bibitem{Sinha13}
Dev~P. Sinha.
\newblock The (non-equivariant) homology of the little disks operad.
\newblock In {\em O{PERADS} 2009}, volume~26 of {\em S\'{e}min. Congr.}, pages
  253--279. Soc. Math. France, Paris, 2013.

\bibitem{Stanley12}
Richard~P. Stanley.
\newblock {\em Enumerative combinatorics. {V}olume 1}, volume~49 of {\em
  Cambridge Studies in Advanced Mathematics}.
\newblock Cambridge University Press, Cambridge, second edition, 2012.

\end{thebibliography}
%\nocite{*}

\end{document}